\documentclass[11pt, reqno]{amsart} \textwidth=14.9cm \oddsidemargin=1cm \evensidemargin=1cm
\usepackage{amsmath,amssymb,amscd,mathrsfs,stmaryrd,wasysym}
\usepackage{amsmath}
\usepackage{amsxtra}
\usepackage{amscd}
\usepackage{amsthm}
\usepackage{amsfonts}
\usepackage{amssymb}
\usepackage{eucal}
\usepackage{color}
\usepackage{graphicx}%
\usepackage{bm}
\newtheorem{theorem}{Theorem}[section]
\newtheorem{lemma}[theorem]{Lemma}
\newtheorem{proposition}[theorem]{Proposition}
\newtheorem{corollary}[theorem]{Corollary}
\theoremstyle{definition}
\newtheorem{definition}[theorem]{Definition}
\newtheorem{remark}[theorem]{Remark}
\newtheorem{example}[theorem]{Example}

\definecolor{A}{rgb}{.75,1,.75}

\numberwithin{equation}{section}

\begin{document}
\title[Cellularity]{Cellularity of cyclotomic Yokonuma-Hecke algebras}
\author[Weideng Cui]{Weideng Cui}
\address{School of Mathematics, Shandong University, Jinan, Shandong 250100, P.R. China.}
\email{cwdeng@amss.ac.cn}

\begin{abstract}
We first give a direct proof of a basis theorem for the cyclotomic Yokonuma-Hecke algebra $Y_{r,n}^{d}(q).$ Our approach follows Kleshchev's, which does not use the representation theory of $Y_{r,n}^{d}(q),$ and so it is very different from that of [ChP2]. We also present two applications. Then we prove that the cyclotomic Yokonuma-Hecke algebra $Y_{r,n}^{d}(q)$ is cellular by constructing an explicit cellular basis, and show that the Jucys-Murphy elements for $Y_{r,n}^{d}(q)$ are JM-elements in the abstract sense. In the appendix, we shall develop the fusion procedure for $Y_{r,n}^{d}(q).$
\end{abstract}

%We establish an equivalence between a module category of the affine (resp. cyclotomic) Yokonuma-Hecke algebra (resp. $Y_{r,n}^{\lambda}(q)$) and its suitable counterpart for a direct sum of tensor products of affine Hecke algebras of type $A$ (resp. cyclotomic Hecke algebras). We then develop several applications of this result. The simple modules of affine Yokonuma-Hecke algebras and of their associated cyclotomic Yokonuma-Hecke algebras are classified over an algebraically closed field of characteristic $p=0$ or $(p,r)=1.$ The modular branching rules for these algebras are obtained, and they are further identified with crystal graphs of integrable modules for affine Lie algebras of type $A.$\end{abstract}

%\thanks{{Keywords}: Affine Yokonuma-Hecke algebras; Cyclotomic Yokonuma-Hecke algebras; Modular branching rules; Crystal graphs} \large

\maketitle
\medskip

\section{Introduction}
\subsection{}
The Yokonuma-Hecke algebra was first introduced by Yokonuma [Yo] as a centralizer algebra associated to the permutation representation of a Chevalley group $G$ with respect to a maximal unipotent subgroup of $G.$ In the 1990s, a new presentation of the Yokonuma-Hecke algebra has been given by Juyumaya [Ju1], and since then it is commonly used for studying this algebra.

The Yokonuma-Hecke algebra $Y_{r,n}(q)$ is a quotient of the group algebra of the modular framed braid group $(\mathbb{Z}/r\mathbb{Z})\wr B_{n},$ where $B_{n}$ is the braid group of type $A$ on $n$ strands. It can also be regraded as a deformation of the group algebra of the complex reflection group $G(r,1,n),$ which is isomorphic to the wreath product $(\mathbb{Z}/r\mathbb{Z})\wr \mathfrak{S}_{n},$ where $\mathfrak{S}_{n}$ is the symmetric group on $n$ letters. It is well-known that there exists another deformation of the group algebra of $G(r,1,n),$ namely the Ariki-Koike algebra $H_{r,n}$ [AK]. The Yokonuma-Hecke algebra $Y_{r,n}(q)$ is quite different from $H_{r,n}.$ For example, the Iwahori-Hecke algebra of type $A$ is canonically a subalgebra of $H_{r,n},$ whereas it is an obvious quotient of $Y_{r,n}(q),$ but not an obvious subalgebra of it.

In the past few years, many people are dedicated to studying the algebra $Y_{r,n}(q)$ from different perspectives. Some impetus comes from knot theory; see the papers [Ju2], [JuL] and [ChL]. In particular, Juyumaya and Kannan [Ju2, JuK] found a basis of $Y_{r,n}(q),$ and then defined a Markov trace on it.

Some other people are particularly interested in the representation theory of $Y_{r,n}(q),$ and also its application to knot theory. Chlouveraki and Poulain d'Andecy [ChPA1] gave explicit formulas for all irreducible representations of $Y_{r,n}(q)$ over $\mathbb{C}(q)$, and obtained a semisimplicity criterion for it. In their subsequent paper [ChPA2], they defined and studied the affine Yokonuma-Hecke algebra $\widehat{Y}_{r,n}(q)$ and the cyclotomic Yokonuma-Hecke algebra $Y_{r,n}^{d}(q),$ and constructed several bases for them, and then showed how to define Markov traces on these algebras. In addition, they gave the classification of irreducible representations of $Y_{r,n}^{d}(q)$ in the generic semisimple case, defined the canonical symmetrizing form on it and computed the associated Schur elements directly.

\subsection{}
Recently, Jacon and Poulain d'Andecy [JaPA] constructed an explicit algebraic isomorphism between the Yokonuma-Hecke algebra $Y_{r,n}(q)$ and a direct sum of matrix algebras over tensor products of Iwahori-Hecke algebras of type $A,$ which is in fact a special case of the results by G. Lusztig [Lu, Section 34].
This allows them to give a description of the modular representation theory of $Y_{r,n}(q)$ and a complete classification of all Markov traces for it. Chlouveraki and S\'{e}cherre [ChS, Theorem 4.3] proved that the affine Yokonuma-Hecke algebra is a particular case of the pro-$p$-Iwahori-Hecke algebra defined by Vign\'eras in [Vi]. Espinoza and Ryom-Hansen [ER] gave a new proof of Jacon and Poulain d'Andecy's isomorphism theorem by giving a concrete isomorphism between $Y_{r,n}(q)$ and Shoji's modified Ariki-Koike algebra $\mathcal{H}_{r,n}.$ Moreover, they showed that $Y_{r,n}(q)$ is a cellular algebra by constructing an explicit cellular basis.

In [CW], we have established an equivalence between a module category of the affine (resp. cyclotomic) Yokonuma-Hecke algebra $\widehat{Y}_{r,n}(q)$ (resp. $Y_{r,n}^{d}(q)$) and its suitable counterpart for a direct sum of tensor products of affine Hecke algebras of type $A$ (resp. cyclotomic Hecke algebras), which allows us to give the classification of simple modules of affine Yokonuma-Hecke algebras and of the associated cyclotomic Yokonuma-Hecke algebras over an algebraically closed field of characteristic $p$ when $p$ does not divide $r,$ and also describe the classification of blocks for these algebras. In addition, the modular branching rules for cyclotomic (resp. affine) Yokonuma-Hecke algebras are obtained, and they are further identified with crystal graphs of integrable modules for affine lie algebras of type $A.$

\subsection{}
Since the cyclotomic Yokonuma-Hecke algebra $Y_{r,n}^{d}(q)$ is a natural generalization of the Yokonuma-Hecke algebra $Y_{r,n}(q),$ it is natural to try to generalize the results of [ER] to the cyclotomic case.

In this paper, We will first follow Kleshchev's approach in [Kle] to give a direct proof of a basis theorem for the cyclotomic Yokonuma-Hecke algebra $Y_{r,n}^{d}(q).$ Our method does not use the representation theory of $Y_{r,n}^{d}(q),$ and so it is very different from that of [ChPA2]. We also present two applications. Then we are largely inspired by the work of [DJM] and [ER] to prove that the cyclotomic Yokonuma-Hecke algebra $Y_{r,n}^{d}(q)$ is a cellular algebra by constructing an explicit cellular basis, and to show that the Jucys-Murphy elements for $Y_{r,n}^{d}(q)$ are JM-elements in Mathas' sense [Ma2].

This paper is organized as follows. In Section 2, we follow Kleshchev's approach in [Kle] to give a direct proof of a basis theorem for the cyclotomic Yokonuma-Hecke algebra $Y_{r,n}^{d}(q).$ In Section 3, we will consider a special case of a Mackey theorem for $Y_{r,n}^{d}(q).$ In Section 4, we will give another proof of the fact that the cyclotomic Yokonuma-Hecke algebra $Y_{r,n}^{d}(q)$ is a Frobenius algebra, and also prove that the induction functor $\mathrm{Ind}_{Y_{r,n}^{d}}^{Y_{r,n+1}^{d}}$ commutes with the $\tau$-duality. In Section 5, we will recall some notations and combinatorial tools that we will need in the sequel. In Section 6, combining the results of [DJM] with those of [ER], we show that the cyclotomic Yokonuma-Hecke algebra $Y_{r,n}^{d}(q)$ is cellular by constructing an explicit cellular basis. In Section 7, we show that the Jucys-Murphy elements for $Y_{r,n}^{d}(q)$ are JM-elements in the abstract sense introduced by Mathas. For the split semisimple $Y_{r,n}^{d}(q)$, we define the idempotents $E_{\mathfrak{t}}$ of $Y_{r,n}^{d}(q)$ and deduce some properties of them by applying the general theory developed in [Ma2, Section 3].

\section{Basis theorem for cyclotomic Yokonuma-Hecke algebras}
\subsection{The definition of $Y_{r, n}^{d}$}
Let $r, n\in \mathbb{N},$ $r\geq1,$ and let $\zeta=e^{2\pi i/r}.$ Let $q$ be an indeterminate. Let $\mathfrak{S}_{n}$ be the symmetric group on $n$ letters, which acts on the set $\{1,2,\ldots,n\}$ on the right by convention.

Let $\mathcal{R}=\mathbb{Z}[\frac{1}{r}][q,q^{-1},\zeta].$ The affine Yokonuma-Hecke algebra $\widehat{Y}_{r,n}=\widehat{Y}_{r,n}(q)$ is an $\mathcal{R}$-associative algebra generated by the elements $t_{1},\ldots,t_{n},g_{1},\ldots,g_{n-1},X_{1}^{\pm 1},$ in which the generators $t_{1},\ldots,t_{n},g_{1},$ $\ldots,g_{n-1}$ satisfy the following relations:
\begin{equation}\label{rel-def-Y1}\begin{array}{rclcl}
g_ig_j\hspace*{-7pt}&=&\hspace*{-7pt}g_jg_i && \mbox{for all $i,j=1,\ldots,n-1$ such that $\vert i-j\vert \geq 2$;}\\[0.1em]
g_ig_{i+1}g_i\hspace*{-7pt}&=&\hspace*{-7pt}g_{i+1}g_ig_{i+1} && \mbox{for all $i=1,\ldots,n-2$;}\\[0.1em]
t_it_j\hspace*{-7pt}&=&\hspace*{-7pt}t_jt_i &&  \mbox{for all $i,j=1,\ldots,n$;}\\[0.1em]
g_it_j\hspace*{-7pt}&=&\hspace*{-7pt}t_{j s_i}g_i && \mbox{for all $i=1,\ldots,n-1$ and $j=1,\ldots,n$;}\\[0.1em]
t_i^r\hspace*{-7pt}&=&\hspace*{-7pt}1 && \mbox{for all $i=1,\ldots,n$;}\\[0.2em]
g_{i}^{2}\hspace*{-7pt}&=&\hspace*{-7pt}1+(q-q^{-1})e_{i}g_{i} && \mbox{for all $i=1,\ldots,n-1$,}
\end{array}
\end{equation}
where $s_{i}$ is the transposition $(i,i+1)$, and for each $1\leq i\leq n-1$,
$$e_{i} :=\frac{1}{r}\sum\limits_{s=0}^{r-1}t_{i}^{s}t_{i+1}^{-s},$$
together with the following relations concerning the generators $X_{1}^{\pm1}$:
\begin{equation}\label{rel-def-Y2}\begin{array}{rclcl}
X_{1}X_{1}^{-1}\hspace*{-7pt}&=&\hspace*{-7pt}X_{1}^{-1}X_{1}=1;&&\\[0.1em]
g_{1}X_{1}g_{1}X_{1}\hspace*{-7pt}&=&\hspace*{-7pt}X_{1}g_{1}X_{1}g_{1};\\[0.1em]
g_{i}X_{1}\hspace*{-7pt}&=&\hspace*{-7pt}X_{1}g_{i} && \mbox{for all $i=2,\ldots,n-1$;}\\[0.1em]
t_{j}X_{1}\hspace*{-7pt}&=&\hspace*{-7pt}X_{1}t_{j} && \mbox{for all $j=1,\ldots,n$.}
\end{array}
\end{equation}

Note that the elements $e_{i}$ are idempotents in $Y_{r, n}^{d}.$ The elements $g_{i}$ are invertible, with the inverse given by
\begin{equation}\label{inverse}
g_{i}^{-1}=g_{i}-(q-q^{-1})e_{i}\quad\mbox{for~all}~i=1,\ldots,n-1.
\end{equation}

Let $w\in \mathfrak{S}_{n},$ and let $w=s_{i_1}\cdots s_{i_{r}}$ be a reduced expression of $w.$ By Matsumoto's lemma, the element $g_{w} :=g_{i_1}g_{i_2}\cdots g_{i_{r}}$ does not depend on the choice of the reduced expression of $w,$ that is, it is well-defined. Let $l$ denote the length function on $\mathfrak{S}_{n}.$ Then we have
\begin{align}\label{multiplication-formula}
g_{i}g_{w}=\begin{cases}g_{s_{i}w}& \hbox {if } l(s_{i}w)>l(w); \\
g_{s_{i}w}+(q-q^{-1})e_{i}g_{w}& \hbox {if } l(s_{i}w)<l(w).
\end{cases}
\end{align}

Let $i, k\in \{1,2,\ldots,n\}$ and set \begin{equation}\label{idempotents}e_{i,k} :=\frac{1}{r}\sum\limits_{s=0}^{r-1}t_{i}^{s}t_{k}^{-s}.\end{equation} Note that $e_{i,i}=1,$ $e_{i,k}=e_{k,i},$ and that $e_{i,i+1}=e_{i}.$ It can be easily checked that
\begin{equation}\label{relations}\begin{array}{rclcl}
e_{i,k}^{2}\hspace*{-7pt}&=&\hspace*{-7pt}e_{i,k} && \mbox{for all $i,k=1,\ldots,n$,}\\[0.1em]
t_{i}e_{j,k}\hspace*{-7pt}&=&\hspace*{-7pt}e_{j,k}t_{i} && \mbox{for all $i,j,k=1,\ldots,n$,}\\[0.1em]
e_{i,j}e_{k,l}\hspace*{-7pt}&=&\hspace*{-7pt}e_{k,l}e_{i,j} &&  \mbox{for all $i,j,k,l=1,\ldots,n$,}\\[0.1em]
e_ie_{k,l}\hspace*{-7pt}&=&\hspace*{-7pt}e_{s_i(k), s_i(l)}e_i && \mbox{for all $i=1,\ldots,n-1$ and $k,l=1,\ldots,n$,}\\[0.1em]
e_{j,k}g_{i}\hspace*{-7pt}&=&\hspace*{-7pt}g_{i}e_{js_{i},ks_{i}} && \mbox{for all $i=1,\ldots,n-1$ and $j,k=1,\ldots,n$.}
\end{array}
\end{equation}
In particular, we have $e_{i}g_{i}=g_{i}e_{i}$ for all $i=1,2,\ldots,n-1.$

We define inductively the following elements in $\widehat{Y}_{r,n}$:
\begin{equation}\label{JM-elements}
X_{i+1} :=g_{i}X_ig_i\quad\mbox{for}~i=1,\ldots,n-1.
\end{equation}
By [ChPA1, Lemma 1] we have, for any $1\leq i\leq n-1$,
\begin{equation}\label{giXj}
g_{i}X_{j}=X_{j}g_{i}\quad\mathrm{for}~j=1,2,\ldots,n~\mathrm{such~that}~j\neq i, i+1.
\end{equation}
Moreover, by [ChPA1, Proposition 1], we have that the elements $t_{1},\ldots, t_{n}, X_{1},\ldots, X_{n}$ form a commutative family, that is,
\begin{equation}\label{commutation-formulae}
xy=yx~~~\mathrm{for~any}~x,y\in \{t_{1},\ldots, t_{n}, X_{1},\ldots, X_{n}\}.
\end{equation}

Let $d\geq 1$ and $v_1,\ldots,v_d$ be some invertible indeterminates. Set $f_{1} :=(X_{1}-v_{1})\cdots (X_{1}-v_{d}).$ Let $\mathcal{J}_{d}$ denote the two-sided ideal of $\widehat{Y}_{r,n}$ generated by $f_{1},$ and define the cyclotomic Yokonuma-Hecke algebra $Y_{r,n}^{d}=Y_{r,n}^{d}(q)$ to be the quotient $$Y_{r,n}^{d}=\widehat{Y}_{r,n}/\mathcal{J}_{d}.$$

\subsection{A basis theorem for $Y_{r, n}^{d}$}
We will often use the following formulae in the rest of this paper.
\begin{lemma}\label{gxxg-xggx}{\rm (See [ChPA2, Lemma 2.15].)} For $a\in \mathbb{Z}_{\geq 0},$ we have
\begin{align}
\label{relations-1}
g_{i}X_{i}^{a}&=X_{i+1}^{a}g_{i}-(q-q^{-1})e_{i}\sum\limits_{k=1}^{a}X_{i}^{a-k}X_{i+1}^{k}  \quad \hspace{1.5mm}\mbox{for all $i=1,\ldots,n-1,$}\\
\label{relations-2}
g_{i}X_{i+1}^{a}&=X_{i}^{a}g_{i}+(q-q^{-1})e_{i}\sum\limits_{k=0}^{a-1}X_{i}^{k}X_{i+1}^{a-k} \qquad \hspace{3mm}\mbox{for all $i=1,\ldots,n-1.$}
\end{align}
\end{lemma}

From the definition of $f_{1},$ we see that $f_{1}$ is a monic polynomial of degree $d.$ Write $$f_{1}=X_{1}^{d}+a_{1}X_{1}^{d-1}+\cdots+a_{d-1}X_{1}+a_{d}.$$
Note that $a_{d}=(-1)^{d}v_{1}\cdots v_{d}$ is invertible. Set $h_{1} :=f_{1},$ and for $i=2,3,\ldots, n,$ define inductively
\begin{equation}\label{fi-gifigi}
f_{i} :=g_{i-1}f_{i-1}g_{i-1}\quad\mathrm{and}\quad h_{i} :=g_{i-1}h_{i-1}g_{i-1}^{-1}.
\end{equation}

The next lemma easily follows from Lemma \ref{gxxg-xggx} by induction.
\begin{lemma}\label{fixid-hiad}
Let $\mathcal{P}_{n}=\mathcal{R}[X_{1}^{\pm1},\ldots,X_{n}^{\pm1}]$ be the algebra of Laurent polynomials in $X_1,\ldots, X_{n},$ which is regarded as a subalgebra of $\widehat{Y}_{r,n}.$ For $i=1,2,\ldots,n,$ we have
\begin{align}
\label{relations-3}
f_{i}&=X_{i}^{d}+(terms~lying~in~\mathcal{P}_{i-1}X_{i}^{e}Y_{r,i}~for~0\leq e< d),\\
\label{relations-4}
h_{i}&=a_{d}+(terms~lying~in~\mathcal{P}_{i-1}X_{i}^{f}Y_{r,i}~for~0< f\leq d).
\end{align}
\end{lemma}

We define $$\Pi_{n} :=\{(\alpha,Z)\:|\:Z\in \{1,2,\ldots,n\},\alpha=(\alpha_{1},\ldots,\alpha_{n})\in \mathbb{Z}^{n}~\mathrm{with}~0\leq \alpha_{i}< d~\mathrm{whenever}~i\notin Z\},$$
$$\Pi_{n}^{+}=\{(\alpha,Z)\:|\: Z\neq \emptyset\}.$$
Given $(\alpha,Z)\in \Pi_{n}$ with $Z=\{z_{1},z_{2},\ldots,z_{k}\},$ we also define $P_{Z} :=P_{z_{1}}P_{z_{2}}\cdots P_{z_{k}},$ where
$$P_{z_{i}}=\begin{cases}f_{z_{i}}& \hbox {if } \alpha_{z_{i}}\geq 0; \\h_{z_{i}}& \hbox {if } \alpha_{z_{i}}< 0.\end{cases}$$
\begin{lemma}\label{free-Yrn-module}
$\widehat{Y}_{r,n}$ is a free right $Y_{r,n}$-module with an $\mathcal{R}$-basis $\{X^{\alpha}P_{Z}\:|\:(\alpha,Z)\in \Pi_{n}\}.$
\end{lemma}
\begin{proof}
Define a lexicographic ordering on $\mathbb{Z}^{n}$: $\alpha\prec\alpha'$ if and only if $$\alpha_{n}=\alpha_{n}',\ldots,\alpha_{k+1}=\alpha_{k+1}', \alpha_{k}<\alpha_{k}'$$ for some $k=1,2,\ldots,n.$ Define a function $\theta: \Pi_{n}\rightarrow \mathbb{Z}^{n}$ by $\theta(\alpha, Z)=(\theta_{1},\ldots,\theta_{n}),$ where

$$\theta_{i}=\begin{cases}\alpha_{i}& ~~\hbox {if } i\notin Z~\mathrm{or}~i\in Z~\mathrm{with}~\alpha_{i}< 0; \\\alpha_{i}+d& ~~\hbox {if } i\in Z~\mathrm{with}~\alpha_{i}\geq 0.\end{cases}$$

Given $(\alpha,Z)\in \Pi_{n}$ with $Z=\{i_{1},\ldots,i_{r},j_{1},\ldots,j_{s}\}$ such that $\alpha_{i_{k}}\geq 0$ ($1\leq k\leq r$) and $\alpha_{j_{l}}< 0$ ($1\leq l\leq s$), we can prove, using induction on $r, s$ and Lemma \ref{fixid-hiad}, that
\begin{equation}\label{XX-PPX}
(X_{i_{1}}^{\alpha_{i_{1}}}\cdots X_{i_{r}}^{\alpha_{i_{r}}})(P_{i_{1}}\cdots P_{i_{r}})=X^{\theta_{r}}+(\mathrm{terms~lying~in~}X^{\beta_{r}}Y_{r,n})
\end{equation}
for $(\alpha_{i_{1}},\ldots,\alpha_{i_{r}})\preceq \beta_{r}=(\beta_{i_{1}},\ldots,\beta_{i_{r}})\prec \theta_{r},$ where $\theta_{r}=(\theta_{i_{1}},\ldots,\theta_{i_{r}})=(\alpha_{i_{1}}+d,\ldots,\alpha_{i_{r}}+d),$ $X^{\theta_{r}}=X_{i_{1}}^{\theta_{i_{1}}}\cdots X_{i_{r}}^{\theta_{i_{r}}},$ and $X^{\beta_{r}}=X_{i_{1}}^{\beta_{i_{1}}}\cdots X_{i_{r}}^{\beta_{i_{r}}};$

Similarly we can prove that
\begin{equation}\label{XX-PPX2}
(X_{j_{1}}^{\alpha_{j_{1}}}\cdots X_{j_{s}}^{\alpha_{j_{r}}})(P_{j_{1}}\cdots P_{j_{s}})=a_{d}^{s}X^{\theta_{s}}+(\mathrm{terms~lying~in~}X^{\gamma_{s}}Y_{r,n})
\end{equation}
for $(\alpha_{j_{1}},\ldots,\alpha_{j_{s}})\prec \gamma_{s}=(\gamma_{j_{1}},\ldots,\gamma_{j_{s}})\preceq (\alpha_{j_{1}}+d,\ldots,\alpha_{j_{s}}+d),$ where $X^{\theta_{s}}=X_{j_{1}}^{\theta_{j_{1}}}\cdots X_{j_{s}}^{\theta_{j_{s}}}=X_{j_{1}}^{\theta_{j_{1}}}\cdots X_{j_{s}}^{\theta_{j_{s}}}$ and $X^{\gamma_{s}}=X_{j_{1}}^{\gamma_{j_{1}}}\cdots X_{j_{s}}^{\gamma_{j_{s}}}.$

Since $\theta: \Pi_{n}\rightarrow \mathbb{Z}^{n}$ is a bijection and we already know that the next set $$\{X^{\alpha}=X_{1}^{\alpha_{1}}\cdots X_{n}^{\alpha_{n}}\:|\:\alpha=(\alpha_{1},\ldots,\alpha_{n})\in \mathbb{Z}^{n}\}$$
is an $\mathcal{R}$-basis for $\widehat{Y}_{r,n}$ viewed as a right $Y_{r,n}$-module by [CW, Theorem 2.3], where the proof can be easily adapted to be true over $\mathcal{R},$ \eqref{XX-PPX} and \eqref{XX-PPX2} imply this lemma.
\end{proof}
\begin{lemma}\label{YfY-fY}
For $n\geq 1,$ we have $Y_{r,n-1}f_{n}Y_{r,n}=f_{n}Y_{r,n}$ and $Y_{r,n-1}h_{n}Y_{r,n}=h_{n}Y_{r,n}.$
\end{lemma}
\begin{proof}
It suffices to prove that the left multiplication by the elements $t_{1},\ldots,t_{n-1},$ $g_{1},\ldots,g_{n-2}$ leaves the space $f_{n}Y_{r,n}$ invariant. Considering $t_{1},\ldots,t_{n-1},$ $g_{1},\ldots,g_{n-3},$ it easily follows from the definition of $f_{n}$ and $h_{n}.$ Considering $g_{n-2},$ we have
\begin{align*}
g_{n-2}f_{n}Y_{r,n}&=g_{n-2}g_{n-1}g_{n-2}f_{n-2}g_{n-2}g_{n-1}Y_{r,n}\\&=g_{n-1}g_{n-2}g_{n-1}f_{n-2}g_{n-2}g_{n-1}Y_{r,n}\\
&=g_{n-1}g_{n-2}f_{n-2}g_{n-1}g_{n-2}g_{n-1}Y_{r,n}\\&=g_{n-1}g_{n-2}f_{n-2}g_{n-2}g_{n-1}g_{n-2}Y_{r,n}\\
&=f_{n}Y_{r,n}.
\end{align*}
It follows from the definition that we have $f_{n}Y_{r,n}=h_{n}Y_{r,n},$ so we have obtained the second equality.
\end{proof}

\begin{lemma}\label{Jd-PNfY}
We have $\mathcal{J}_{d}=\sum\limits_{i=1}^{n}\mathcal{P}_{n}f_{i}Y_{r,n}.$
\end{lemma}
\begin{proof}
\begin{align*}
\mathcal{J}_{d}&=\widehat{Y}_{r,n}f_{1}\widehat{Y}_{r,n}=\widehat{Y}_{r,n}f_{1}\mathcal{P}_{n}Y_{r,n}=\widehat{Y}_{r,n}f_{1}Y_{r,n}\\
&=\mathcal{P}_{n}Y_{r,n}f_{1}Y_{r,n}=\sum\limits_{i=1}^{n}\sum\limits_{j=0}^{r-1}\sum\limits_{u\in Y_{r,2\ldots n}}\mathcal{P}_{n}t_{i}^{j}g_{i-1}\cdots g_{1}uf_{1}Y_{r,n}\\
&=\sum\limits_{i=1}^{n}\mathcal{P}_{n}g_{i-1}\cdots g_{1}f_{1}Y_{r,n}=\sum\limits_{i=1}^{n}\mathcal{P}_{n}f_{i}Y_{r,n}.
\end{align*}
as required.
\end{proof}

\begin{lemma}\label{Jdsum-ZPiZPz}  For $d\geq 1,$ we have $\mathcal{J}_{d}=\sum\limits_{(\alpha, Z)\in \Pi_{n}^{+}}X^{\alpha}P_{Z}Y_{r,n}.$\end{lemma}
\begin{proof}
Proceed by induction on $n,$ the case $n=1$ being obvious. Let $n>1.$ Let $\mathcal{J}_{d}'=\widehat{Y}_{r,n-1}f_{1}\widehat{Y}_{r,n-1}.$ We have
\begin{equation}\label{JdPi-XPY}
\mathcal{J}_{d}'=\sum\limits_{(\alpha, Z)\in \Pi_{n-1}^{+}}X^{\alpha}P_{Z}Y_{r,n-1}
\end{equation}
by the induction hypothesis. Let $$\mathcal{J}=\sum\limits_{(\alpha, Z)\in \Pi_{n}^{+}}X^{\alpha}P_{Z}Y_{r,n}.$$ Obviously $\mathcal{J}\subseteq \mathcal{J}_{d}.$ So in view of Lemma \ref{Jd-PNfY}, it suffices to prove that $X^{\alpha}f_{i}Y_{r,n}\subseteq \mathcal{J}$ for each $\alpha\in \mathbb{Z}^{n}$ and each $i=1,2,\ldots,n.$

Consider first $X^{\alpha}f_{n}Y_{r,n}.$ Write $X^{\alpha}=X_{n}^{\alpha_{n}}X^{\beta}$ for $\beta\in \mathbb{Z}^{n-1}.$ Expanding $X^{\beta}$ in terms of the basis of $\widehat{Y}_{r,n-1}$ from Lemma \ref{free-Yrn-module}, we see that, when $\alpha_{n}\geq 0,$
$$X^{\alpha}f_{n}Y_{r,n}\subseteq \sum_{(\alpha', Z')\in \Pi_{n-1}}X_{n}^{\alpha_{n}}X^{\alpha'}P_{Z'}Y_{r,n-1}f_{n}Y_{r,n};$$
when $\alpha_{n}< 0,$
$$X^{\alpha}f_{n}Y_{r,n}=X^{\alpha}h_{n}Y_{r,n}\subseteq \sum_{(\alpha'', Z'')\in \Pi_{n-1}}X_{n}^{\alpha_{n}}X^{\alpha''}P_{Z''}Y_{r,n-1}h_{n}Y_{r,n}.$$
It is always contained in $\mathcal{J}$ thanks to Lemma \ref{YfY-fY}.

The next, we will consider $X^{\alpha}f_{i}Y_{r,n}$ with $i< n.$ Write $X^{\alpha}=X_{n}^{\alpha_{n}}X^{\beta}$ for $\beta\in \mathbb{Z}^{n-1}.$ By the induction hypothesis, we have
$$X^{\alpha}f_{i}Y_{r,n}=X_{n}^{\alpha_{n}}X^{\beta}f_{i}Y_{r,n}\subseteq \sum_{(\alpha', Z')\in \Pi_{n-1}^{+}}X_{n}^{\alpha_{n}}X^{\alpha'}P_{Z'}Y_{r,n}.$$

When $\alpha_{n}\geq 0,$ we will show by induction on $\alpha_{n}$ that $X_{n}^{\alpha_{n}}X^{\alpha'}P_{Z'}Y_{r,n}\in \mathcal{J}$ for each $(\alpha', Z')\in \Pi_{n-1}^{+}.$ This is immediate if $0\leq \alpha_{n}< d.$ Assume that $\alpha_{n}\geq d.$ Expanding $f_{n}$ using Lemma \ref{fixid-hiad}, the set $$X_{n}^{\alpha_{n}-d}X^{\alpha'}P_{Z'}f_{n}Y_{r,n}\subseteq \mathcal{J}$$ looks like the desired $X_{n}^{\alpha_{n}}X^{\alpha'}P_{Z'}Y_{r,n}$ plus a sum of terms belonging to $X_{n}^{\alpha_{n}-d+e}\mathcal{J}_{d}'Y_{r,n}$ with $0\leq e< d.$ It now suffices to show that each such term $X_{n}^{\alpha_{n}-d+e}\mathcal{J}_{d}'Y_{r,n}\subseteq \mathcal{J}.$ But by \eqref{JdPi-XPY},
$$X_{n}^{\alpha_{n}-d+e}\mathcal{J}_{d}'Y_{r,n}\subseteq \sum_{(\alpha', Z')\in \Pi_{n-1}^{+}}X_{n}^{\alpha_{n}-d+e}X^{\alpha'}P_{Z'}Y_{r,n}$$
and each term in the summation lies in $\mathcal{J}$ by induction, since $0\leq \alpha_{n}-d+e<\alpha_{n}.$

When $\alpha_{n}< 0,$ we will similarly show by induction on $\alpha_{n}$ that $X_{n}^{\alpha_{n}}X^{\alpha'}P_{Z'}Y_{r,n}\in \mathcal{J}$ for each $(\alpha', Z')\in \Pi_{n-1}^{+}.$ This is obvious if $0\leq \alpha_{n}< d.$ Assume that $\alpha_{n}< 0.$ Expanding $h_{n}$ using Lemma \ref{fixid-hiad}, the set
$$a_{d}^{-1}X_{n}^{\alpha_{n}}X^{\alpha'}P_{Z'}h_{n}Y_{r,n}\subseteq \mathcal{J}$$ looks like the desired $X_{n}^{\alpha_{n}}X^{\alpha'}P_{Z'}Y_{r,n}$ plus a sum of terms belonging to $X_{n}^{\alpha_{n}+f}\mathcal{J}_{d}'Y_{r,n}$ with $0< f\leq d.$ It now suffices to show that each such term $X_{n}^{\alpha_{n}+f}\mathcal{J}_{d}'Y_{r,n}\subseteq \mathcal{J}.$ But by \eqref{JdPi-XPY},
$$X_{n}^{\alpha_{n}+f}\mathcal{J}_{d}'Y_{r,n}\subseteq \sum_{(\alpha', Z')\in \Pi_{n-1}^{+}}X_{n}^{\alpha_{n}+f}X^{\alpha'}P_{Z'}Y_{r,n}$$
and each term in the summation lies in $\mathcal{J}$ by induction, since $\alpha_{n}< \alpha_{n}+f\leq d.$
\end{proof}
\begin{theorem}\label{basis-theorem-cyclotomic}
The canonical images of the elements $$\{X^{\alpha}t^{\beta}g_{w}\:|\:\alpha\in \mathbb{Z}_{+}^{n}~with~\alpha_{1},\ldots,\alpha_{n}< d, ~\beta\in \mathbb{Z}_{+}^{n}~with~\beta_{1},\ldots,\beta_{n}< r, ~w\in \mathfrak{S}_{n}\}$$
form an $\mathcal{R}$-basis for $Y_{r,n}^{d}.$
\end{theorem}
\begin{proof}
By Lemma \ref{free-Yrn-module} and \ref{Jdsum-ZPiZPz}, the elements $\{X^{\alpha}P_{Z}\:|\:(\alpha, Z)\in \Pi_{n}^{+}\}$ form an $\mathcal{R}$-basis for $\mathcal{J}_{d}$ viewed as a right $Y_{r,n}$-module. Hence Lemma \ref{free-Yrn-module} implies that the elements
$$\{X^{\alpha}\:|\:\alpha\in \mathbb{Z}_{+}^{n}~\mathrm{with}~\alpha_{1},\ldots,\alpha_{n}< d\}$$ form an $\mathcal{R}$-basis for a complement to $\mathcal{J}_{d}$ in $\widehat{Y}_{r,n}$ viewed as a right $Y_{r,n}$-module. We immediately get the theorem.
\end{proof}

\section{Mackey theorem for cyclotomic Yokonuma-Hecke algebras}

In this section, we shall consider a special case of a Mackey theorem for $Y_{r,n}^{d}.$ Given any $y\in \widehat{Y}_{r,n},$ we will denote its canonical image in $Y_{r,n}^{d}$ by the same symbol. Thus, Theorem \ref{basis-theorem-cyclotomic} says that $$\{X^{\alpha}t^{\beta}g_{w}\:|\:\alpha\in \mathbb{Z}_{+}^{n}~\mathrm{with}~\alpha_{1},\ldots,\alpha_{n}< d, ~\beta\in \mathbb{Z}_{+}^{n}~\mathrm{with}~\beta_{1},\ldots,\beta_{n}< r, ~w\in \mathfrak{S}_{n}\}$$ is an $\mathcal{R}$-basis for $Y_{r,n}^{d}.$ Also, Theorem \ref{basis-theorem-cyclotomic} implies that $Y_{r,n}^{d}$ is a subalgebra of $Y_{r,n+1}^{d}.$ So we can define the induction functor $\mathrm{Ind}_{Y_{r,n}^{d}}^{Y_{r,n+1}^{d}}$ and the restriction functor $\mathrm{Res}_{Y_{r,n}^{d}}^{Y_{r,n+1}^{d}}.$ We define, for each $Y_{r,n}^{d}$-module $M,$ $$\mathrm{Ind}_{Y_{r,n}^{d}}^{Y_{r,n+1}^{d}}M :=Y_{r,n+1}^{d}\otimes_{Y_{r,n}^{d}}M.$$
\begin{lemma}\label{3-a-b-c}
$(a)$ $Y_{r,n+1}^{d}$ is a free right $Y_{r,n}^{d}$-module with basis $$\{X_{j}^{a}t_{j}^{b}g_{j}\cdots g_{n}\:|\:0\leq a<d,~0\leq b< r,~1\leq j\leq n+1\}.$$

$(b)$ As $(Y_{r,n}^{d}, Y_{r,n}^{d})$-bimodules, we have $$Y_{r,n+1}^{d}=Y_{r,n}^{d}g_{n}Y_{r,n}^{d}\oplus \bigoplus_{0\leq a<d}\bigoplus_{0\leq b< r}X_{n+1}^{a}t_{n+1}^{b}Y_{r,n}^{d}.$$

$(c)$ For $0\leq a<d,$ $0\leq b< r,$ there are isomorphisms $$Y_{r,n}^{d}g_{n}Y_{r,n}^{d}\cong Y_{r,n}^{d}\otimes_{Y_{r,n-1}^{d}} Y_{r,n}^{d}\quad and \quad X_{n+1}^{a}Y_{r,n}^{d}\cong Y_{r,n}^{d},~~t_{n+1}^{b}Y_{r,n}^{d}\cong Y_{r,n}^{d}$$
of $(Y_{r,n}^{d}, Y_{r,n}^{d})$-bimodules.
\end{lemma}
\begin{proof}
$(a)$ By Theorem \ref{basis-theorem-cyclotomic} and dimension considerations, we just need to check that $Y_{r,n+1}^{d}$ is generated as a right $Y_{r,n}^{d}$-module by the given elements. This follows using Lemma \ref{gxxg-xggx}.

$(b)$ It suffices to notice, using $(a)$ and Lemma \ref{gxxg-xggx}, that $$\{X_{j}^{a}t_{j}^{b}g_{j}\cdots g_{n}\:|\:0\leq a<d,~0\leq b< r,~1\leq j\leq n\}$$
is a basis of $Y_{r,n}^{d}g_{n}Y_{r,n}^{d}$ as a free right $Y_{r,n}^{d}$-module.

$(c)$ The isomorphisms $X_{n+1}^{a}Y_{r,n}^{d}\cong Y_{r,n}^{d}$ and $t_{n+1}^{b}Y_{r,n}^{d}\cong Y_{r,n}^{d}$ are clear from $(a).$ Furthermore, the map
$$Y_{r,n}^{d}\times Y_{r,n}^{d}\rightarrow Y_{r,n}^{d}g_{n}Y_{r,n}^{d},~~(u, v)\mapsto ug_{n}v$$
is $Y_{r,n-1}^{d}$-balanced, since $g_{n}$ centralizes $Y_{r,n-1}^{d}.$ So it induces a homomorphism
$$\Phi: Y_{r,n}^{d}\otimes_{Y_{r,n-1}^{d}} Y_{r,n}^{d}\rightarrow Y_{r,n}^{d}g_{n}Y_{r,n}^{d}$$
of $(Y_{r,n}^{d}, Y_{r,n}^{d})$-bimodules. By $(a),$ $Y_{r,n}^{d}\otimes_{Y_{r,n-1}^{d}} Y_{r,n}^{d}$ is a free right $Y_{r,n}^{d}$-module with basis $$\{X_{j}^{a}t_{j}^{b}g_{j}\cdots g_{n-1}\otimes 1\:|\:0\leq a<d,~0\leq b< r,~1\leq j\leq n\}.$$
But $\Phi$ maps these elements to a basis for $Y_{r,n}^{d}g_{n}Y_{r,n}^{d}$ as a free right $Y_{r,n}^{d}$-module, using a fact observed in the proof of $(b)$.
This shows that $\Phi$ is an isomorphism.
\end{proof}

We have now decomposed $Y_{r,n+1}^{d}$ as a $(Y_{r,n}^{d}, Y_{r,n}^{d})$-bimodules. Using Lemma \ref{3-a-b-c}$(b)$ and $(c),$ the standard argument yields the following result.
\begin{theorem}\label{3-2-res}
Let $M$ be a $Y_{r,n}^{d}$-module. Then there is a natural isomorphism
$$\mathrm{Res}_{Y_{r,n}^{d}}^{Y_{r,n+1}^{d}}\mathrm{Ind}_{Y_{r,n}^{d}}^{Y_{r,n+1}^{d}}M\cong M^{\oplus rd}\oplus \mathrm{Ind}_{Y_{r,n}^{d}}^{Y_{r,n+1}^{d}}\mathrm{Res}_{Y_{r,n}^{d}}^{Y_{r,n+1}^{d}}M$$
of $Y_{r,n}^{d}$-modules.
\end{theorem}

\section{Duality for cyclotomic Yokonuma-Hecke algebras}

$\widehat{Y}_{r,n}$ possesses an anti-automorphism $\tau$ defined on generators as follows:
$$\tau: g_{i}\mapsto g_{i},~~~~X_{j}\mapsto X_{j},~~~~t_{j}\mapsto t_{j}$$
for all $i=1,\ldots,n-1,$ $j=1,\ldots,n.$ If $M$ is a finite dimensional $\widehat{Y}_{r,n}$-module, we can use $\tau$ to make the dual space $M^{*}$ into a $\widehat{Y}_{r,n}$-module denoted by $M^{\tau}.$ Since $\tau$ leaves the two-sided ideal $\mathcal{J}_{d}$ invariant, it induces a duality also denoted by $\tau$ on $Y_{r,n}^{d}$ and on finite dimensional $Y_{r,n}^{d}$-modules.

In this section, we will prove that the induction functor $\mathrm{Ind}_{Y_{r,n}^{d}}^{Y_{r,n+1}^{d}}$ commutes with the $\tau$-duality. Let us first give some preliminary work.
\begin{lemma}\label{4-1-gXI-gign}
For $1\leq i\leq n$ and $a\geq 0,$ we have
$$g_{n}\cdots g_{i}X_{i}^{a}g_{i}\cdots g_{n}=X_{n+1}^{a}+(terms~lying~in~Y_{r,n}^{d}g_{n}Y_{r,n}^{d}+\sum\limits_{k=1}^{a-1}\sum\limits_{s=0}^{r-1}X_{n+1}^{k}t_{n+1}^{s}Y_{r,n}^{d}).$$
\end{lemma}
\begin{proof}
We prove it by induction on $n=i, i+1,\ldots.$ In case $n=i,$
\begin{align*}
g_{i}X_{i}^{a}g_{i}&=X_{i+1}^{a}[1+(q-q^{-1}e_{i}g_{i})]-(q-q^{-1})e_{i}\sum\limits_{k=1}^{a-1}X_{i}^{a-k}X_{i+1}^{k}g_{i}-(q-q^{-1})e_{i}X_{i+1}^{a}g_{i}\\
&=X_{i+1}^{a}-(q-q^{-1})e_{i}\sum\limits_{k=1}^{a-1}X_{i}^{a-k}[g_{i}X_{i}^{k}+(q-q^{-1})e_{i}\sum\limits_{l=1}^{k}X_{i}^{k-l}X_{i+1}^{l}]\\
&=X_{i+1}^{a}-(q-q^{-1})\frac{1}{r}\sum\limits_{k=1}^{a-1}\sum\limits_{s=0}^{r-1}X_{i}^{a-k}t_{i}^{s}g_{i}t_{i}^{-s}X_{i}^{k}\\
&~~~-(q-q^{-1})^{2}\frac{1}{r}\sum\limits_{k=1}^{a-1}\sum\limits_{l=1}^{k}\sum\limits_{s=0}^{r-1}X_{i+1}^{l}t_{i+1}^{s}X_{i}^{a-l}t_{i}^{-s},
\end{align*}
the result follows from the calculation above. The induction step is similar, noting that $g_{n}$ centralizes $Y_{r,n-1}^{d}.$
\end{proof}

From Lemma \ref{4-1-gXI-gign}, we can easily get the following result.
\begin{lemma}\label{Xn1-adYrngnyrn}
We have $$X_{n+1}^{d}=-a_{d}+(terms~lying~in~Y_{r,n}^{d}g_{n}Y_{r,n}^{d}+\sum\limits_{k=1}^{d-1}\sum\limits_{s=0}^{r-1}X_{n+1}^{k}t_{n+1}^{s}Y_{r,n}^{d}).$$
\end{lemma}
\begin{lemma}\label{4-3-yrnd}
There exists a $(Y_{r,n}^{d}, Y_{r,n}^{d})$-bimodule homomorphism $\theta: Y_{r,n+1}^{d}\rightarrow Y_{r,n}^{d}$ such that $\mathrm{ker}~\theta$ contains no non-zero left ideals of $Y_{r,n+1}^{d}.$
\end{lemma}
\begin{proof}
By Lemma \ref{3-a-b-c}$(b),$ we know that

$$Y_{r,n+1}^{d}=Y_{r,n}^{d}\oplus \bigoplus_{b=1}^{r-1}t_{n+1}^{b}Y_{r,n}^{d}\oplus \bigoplus_{a=1}^{d-1}\bigoplus_{b=0}^{r-1}X_{n+1}^{a}t_{n+1}^{b}Y_{r,n}^{d}\oplus Y_{r,n}^{d}g_{n}Y_{r,n}^{d}$$ as a $(Y_{r,n}^{d}, Y_{r,n}^{d})$-bimodule. Let $\theta: Y_{r,n+1}^{d}\rightarrow Y_{r,n}^{d}$ be the projection on to the first summand of this bimodule decomposition, which by Lemma \ref{3-a-b-c}$(c)$ is a $(Y_{r,n}^{d}, Y_{r,n}^{d})$-bimodule homomorphism. So it suffices to prove that if $y\in Y_{r,n+1}^{d}$ has the property that $\theta(hy)=0$ for all $h\in Y_{r,n+1}^{d},$ then $y=0.$ By Lemma \ref{3-a-b-c}$(a),$ we may write
$$y=\sum\limits_{a=0}^{d-1}\sum\limits_{b=0}^{r-1}X_{n+1}^{a}t_{n+1}^{b}t_{a,b}+\sum\limits_{a=0}^{d-1}\sum\limits_{b=0}^{r-1}
\sum\limits_{j=1}^{n}X_{j}^{a}t_{j}^{b}g_{j}\cdots g_{n}u_{a,b,j}$$
for some elements $t_{a,b},$ $u_{a,b,j}\in Y_{r,n}^{d}.$

As $\theta(y)=0,$ we must have $t_{0,0}=0.$ Now $\theta(X_{n+1}y)=0$ implies that $-a_{d}t_{d-1,0}=0$ by Lemma \ref{Xn1-adYrngnyrn}, that is, $t_{d-1,0}=0.$ Similarly, $t_{d-2,0}=\cdots=t_{1,0}=0.$ As $\theta(t_{n+1}^{-1}y)=0,$ we must have $t_{0,1}=0.$ Now $\theta(X_{n+1}t_{n+1}^{-1}y)=0$ implies that $-a_{d}t_{d-1,1}=0$ by Lemma \ref{Xn1-adYrngnyrn}, that is, $t_{d-1,1}=0.$ Similarly, $t_{d-2,1}=\cdots=t_{1,1}=0.$ Similarly, we have $t_{0,2}=t_{d-1,2}=\cdots=t_{1,2}=\cdots=t_{0,r-1}=t_{d-1,r-1}=\cdots=t_{1,r-1}=0.$

Next, considering $\theta(g_{n}y)=0,$ we get $u_{0,0,n}=0$ by Lemma \ref{4-1-gXI-gign}. Since $\theta(X_{n+1}g_{n}y)=0,$ we get $-a_{d}u_{d-1,0,n}=0$ by Lemma \ref{4-1-gXI-gign} and \ref{Xn1-adYrngnyrn}, that is, $u_{d-1,0,n}=0.$ Considering $\theta(X_{n+1}^{2}g_{n}y)=\cdots=\theta(X_{n+1}^{d-1}g_{n}y)=0,$ we similarly get that $u_{d-2,0,n}=\cdots=u_{1,0,n}=0.0.$ Since $\theta(t_{n+1}^{-1}g_{n}y)=0,$ we get $u_{0,1,n}=0.$ Considering $\theta(X_{n+1}t_{n+1}^{-1}g_{n}y)=\theta(X_{n+1}^{2}t_{n+1}^{-1}g_{n}y)=\cdots=\theta(X_{n+1}^{d-1}t_{n+1}^{-1}g_{n}y)=0,$ we get that $-a_{d}u_{d-1,1,n}=-a_{d}u_{d-2,1,n}=\cdots=-a_{d}u_{1,1,n}=0,$ that is, $u_{d-1,1,n}=u_{d-2,1,n}=\cdots=u_{1,1,n}=0.$ Next, considering $\theta(t_{n+1}^{-2}g_{n}y)=\theta(X_{n+1}t_{n+1}^{-2}g_{n}y)=\cdots=\theta(X_{n+1}^{d-1}t_{n+1}^{-2}g_{n}y)=0,\ldots,
\theta(t_{n+1}^{-(r-1)}g_{n}y)=\theta(X_{n+1}t_{n+1}^{-(r-1)}g_{n}y)=\cdots=\theta(X_{n+1}^{d-1}t_{n+1}^{-(r-1)}g_{n}y)=0,$ we get that $u_{0,2,n}=u_{d-1,2,n}=\cdots=u_{1,2,n}=0,\ldots,u_{0,r-1,n}=u_{d-1,r-1,n}=\cdots=u_{1,r-1,n}=0.$

Repeat the argument again, this time considering $\theta(g_{n}g_{n-1}y), \theta(X_{n+1}g_{n}g_{n-1}y),\ldots,$ $\theta(X_{n+1}^{d-1}g_{n}g_{n-1}y), \theta(t_{n+1}^{-1}g_{n}g_{n-1}y), \theta(X_{n+1}t_{n+1}^{-1}g_{n}g_{n-1}y), \theta(X_{n+1}^{d-1}t_{n+1}^{-1}g_{n}g_{n-1}y),\ldots, \theta(t_{n+1}^{-(r-1)}$ $g_{n}g_{n-1}y),\ldots,  \theta(X_{n+1}^{d-1}t_{n+1}^{-(r-1)}g_{n}g_{n-1}y),$ to get that all $u_{a,b,n-1}=0$ for $0\leq a\leq d-1$ and $0\leq b\leq r-1.$ Continuing in this way we eventually arrive at the desired conclusion that $y=0.$
\end{proof}

Now we will prove the main result of this section.
\begin{theorem}\label{theorem4-4}
There exists a natural isomorphism $$Y_{r,n+1}^{d}\otimes_{Y_{r,n}^{d}}M\cong \mathrm{Hom}_{Y_{r,n}^{d}}(Y_{r,n+1}^{d}, M)$$
for all $Y_{r,n}^{d}$-modules $M.$
\end{theorem}
\begin{proof}
We show that there exists an isomorphism
$$\varphi: Y_{r,n+1}^{d}\rightarrow \mathrm{Hom}_{Y_{r,n}^{d}}(Y_{r,n+1}^{d}, Y_{r,n}^{d})$$
of $(Y_{r,n+1}^{d}, Y_{r,n}^{d})$-bimodules. Then, applying the functor $?\otimes_{Y_{r,n}^{d}}M,$ we obtain natural isomorphisms
$$Y_{r,n+1}^{d}\otimes_{Y_{r,n}^{d}}M\overset{\varphi\otimes\mathrm{id}}{\longrightarrow}\mathrm{Hom}_{Y_{r,n}^{d}}(Y_{r,n+1}^{d}, Y_{r,n}^{d})\otimes_{Y_{r,n}^{d}}M\cong \mathrm{Hom}_{Y_{r,n}^{d}}(Y_{r,n+1}^{d}, M)$$
as required. Note that the existence of the second isomorphism here follows from the fact that $Y_{r,n+1}^{d}$ is a projective left $Y_{r,n}^{d}$-module and [AF, 20.10].

To construct $\varphi,$ let $\theta$ be as in Lemma \ref{4-3-yrnd}, and define $\varphi(h)$ to be the map $\theta_{h},$ for each $h\in Y_{r,n+1}^{d},$ where
$$\theta_{h}: Y_{r,n+1}^{d}\rightarrow Y_{r,n}^{d},~~~h'\mapsto \theta(h'h).$$
We can easily check that $\varphi,$ which is defined above, is then a well-defined homomorphism of $(Y_{r,n+1}^{d}, Y_{r,n}^{d})$-bimodules. To see that it is an isomorphism, it suffices by dimension considerations to check that it is injective. If $\varphi(h)=0$ for some $h\in Y_{r,n+1}^{d},$ then for every $x\in Y_{r,n+1}^{d},$ $\theta(xh)=0,$ that is the left ideal $Y_{r,n+1}^{d}h$ is contained in $\mathrm{Ker}~\theta.$ So Lemma \ref{4-3-yrnd} implies $h=0.$
\end{proof}

\begin{corollary}\label{corollary4-5-Frobenius}
$Y_{r,n}^{d}$ is a Frobenius algebra, that is, there is an isomorphism of left $Y_{r,n}^{d}$-modules $Y_{r,n}^{d}\cong \mathrm{Hom}_{\mathcal{R}}(Y_{r,n}^{d}, \mathcal{R})$ between the left regular module and the $\mathcal{R}$-linear dual of the right regular module.
\end{corollary}
\begin{proof}
Proceed by induction on $n.$ In case $n=1,$ it is obvious. For the induction step,
\begin{align*}
Y_{r,n}^{d}&\cong Y_{r,n}^{d}\otimes_{Y_{r,n-1}^{d}}Y_{r,n-1}^{d}\cong Y_{r,n}^{d}\otimes_{Y_{r,n-1}^{d}} \mathrm{Hom}_{\mathcal{R}}(Y_{r,n-1}^{d}, \mathcal{R})\\
&\cong \mathrm{Hom}_{Y_{r,n-1}^{d}}(Y_{r,n}^{d}, \mathrm{Hom}_{\mathcal{R}}(Y_{r,n-1}^{d}, \mathcal{R}))\cong \mathrm{Hom}_{\mathcal{R}}(Y_{r,n-1}^{d}\otimes_{Y_{r,n-1}^{d}}Y_{r,n}^{d}, \mathcal{R})\\
&\cong \mathrm{Hom}_{\mathcal{R}}(Y_{r,n}^{d}, \mathcal{R}),
\end{align*}
applying Theorem \ref{theorem4-4} and adjointness of $\otimes$ and $\mathrm{Hom}.$
\end{proof}

The following result says that the induction functor $\mathrm{Ind}_{Y_{r,n}^{d}}^{Y_{r,n+1}^{d}}$ commutes with $\tau.$
\begin{corollary}\label{corollarty4-6-exact}
The exact functor $\mathrm{Ind}_{Y_{r,n}^{d}}^{Y_{r,n+1}^{d}}$ is both left and right adjoint to $\mathrm{Res}_{Y_{r,n}^{d}}^{Y_{r,n+1}^{d}}.$ Moreover, it commutes with the $\tau$-duality in the sense that there is a natural isomorphism
$$\mathrm{Ind}_{Y_{r,n}^{d}}^{Y_{r,n+1}^{d}}(M^{\tau})\cong (\mathrm{Ind}_{Y_{r,n}^{d}}^{Y_{r,n+1}^{d}}M)^{\tau}$$
for all finite dimensional $Y_{r,n}^{d}$-modules $M.$
\end{corollary}
\begin{proof}
The fact that $\mathrm{Ind}_{Y_{r,n}^{d}}^{Y_{r,n+1}^{d}}=Y_{r,n+1}^{d}\otimes_{Y_{r,n}^{d}}?$ is right adjoint to $\mathrm{Res}_{Y_{r,n}^{d}}^{Y_{r,n+1}^{d}}$ is immediate from Theorem \ref{theorem4-4}, since $\mathrm{Hom}_{Y_{r,n}^{d}}(Y_{r,n+1}^{d}, ?)$ is right adjoint to restriction by adjointness of $\otimes$ and $\mathrm{Hom}.$ But on finite dimensional modules, a standard check shows that the functor $\tau\circ\mathrm{Ind}_{Y_{r,n}^{d}}^{Y_{r,n+1}^{d}}\circ\tau$ is also right adjoint to restriction. Now the commutativity of $\mathrm{Ind}_{Y_{r,n}^{d}}^{Y_{r,n+1}^{d}}$ and $\tau$ follows by uniqueness of adjoint functors.
\end{proof}

\section{Combinatorics}

In this section we will review the notations and combinatorial tools that we will need in the sequel.

Let $\mathbf{s}=\{1,2,\ldots,n\}.$ For any nonempty subset $I\subseteq \mathbf{s}$ we define the following element $E_{I}$ by $$E_{I} :=\prod_{i, j\in I; i<j}e_{i,j},$$where by convention $E_{I}=1$ if $|I|=1.$

We also need a further generalization of this. We say that the set $A=\{I_{1}, I_{2},\ldots,I_{k}\}$ is a set partition of $\mathbf{s}$ if the $I_{j}$'s are nonempty and disjoint subsets of $\mathbf{s},$ and their union is $\mathbf{s}.$ We refer to them as the blocks of $A.$ We denote by $\mathcal{SP}_{n}$ the set of all set partitions of $\mathbf{s}.$ For $A=\{I_{1}, I_{2},\ldots,I_{k}\}\in \mathcal{SP}_{n}$ we then define $E_{A} :=\prod_{j}E_{I_{j}}.$

We extend the right action of $\mathfrak{S}_{n}$ on $\mathbf{s}$ to a right action on $\mathcal{SP}_{n}$ by defining $Aw :=\{I_{1}w,\ldots,I_{k}w\}\in \mathcal{SP}_{n}$ for $w\in \mathfrak{S}_{n}.$ Then we can easily get the following lemma.
\begin{lemma}\label{5-1-lemmaaa}
For $A\in \mathcal{SP}_{n}$ and $w\in \mathfrak{S}_{n},$ we have $$g_{w}E_{A}=E_{Aw^{-1}}g_{w}.$$
In particular, if $w$ leaves invariant every block of $A,$ or more generally permutes some of the blocks of $A,$ then $g_{w}$ commutes with $E_{A}.$
\end{lemma}

$\mu=(\mu_{1},\ldots,\mu_{k})$ is called a composition of $n$ if it is a finite sequence of nonnegative integers whose sum is $n.$ A composition $\mu$ is a partition of $n$ if its parts are non-increasing. We write $\mu \models n$ (resp. $\lambda\vdash n$) if $\mu$ is a composition (resp. partition) of $n,$ and we define $|\mu| :=n$ (resp. $|\lambda| :=n$).

We associate a Young diagram to a composition $\mu,$ which is the set $$[\mu] :=\{(i,j)\:|\:i\geq 1~\mathrm{and}~1\leq j\leq \mu_{i}\}.$$ We will regard $[\mu]$ as an array of boxes, or nodes, in the plane. For $\mu \models n,$ we define a $\mu$-tableau by replacing each node of $[\mu]$ by one of the integers $1,2,\ldots,n,$ allowing no repeats.

For $\mu \models n,$ we say that a $\mu$-tableau $\mathfrak{t}$ is row standard if the entries in each row of $\mathfrak{t}$ increase from left to right. A $\mu$-tableau $\mathfrak{t}$ is standard if $\mu$ is a partition, $\mathfrak{t}$ is row standard and the entries in each column increase from top to bottom. For a composition $\mu$ of $n,$ we denote by $\mathfrak{t}^{\mu}$ the $\mu$-tableau in which $1,2,\ldots,n$ appear in increasing order from left to right along the rows of $[\mu].$

The symmetric group $\mathfrak{S}_{n}$ acts from the right on the set of $\mu$-tableaux by permuting the entries in each tableau. For any composition $\mu=(\mu_{1},\ldots,\mu_{k})$ of $n$ we define the Young subgroup $\mathfrak{S}_{\mu} :=\mathfrak{S}_{\mu_{1}}\times\cdots\times\mathfrak{S}_{\mu_{k}},$ which is the row stabilizer of $\mathfrak{t}^{\mu}.$

Let $\lambda=(\lambda_{1},\ldots,\lambda_{k})$ and $\mu=(\mu_{1},\ldots,\mu_{l})$ be two compositions of $n.$ We say that $\lambda\unrhd\mu$ if $$\sum_{i=1}^{j}\lambda_{i}\geq \sum_{i=1}^{j}\mu_{i}~~~~~~\mathrm{for~all~}j\geq 1.$$ If $\lambda\unrhd\mu$ and $\lambda\neq\mu,$ we write $\lambda\rhd\mu.$

We extend the partial order above to tableaux as follows. If $\mathfrak{v}$ is a row standard $\lambda$-tableau and $1\leq k\leq n,$ then the entries $1,2,\ldots,k$ in $\mathfrak{v}$ occupy the diagram of a composition; let $\mathfrak{v}_{\downarrow k}$ denote this composition. Let $\lambda$ and $\mu$ be two compositions of $n.$ Suppose that $\mathfrak{s}$ is a row standard $\lambda$-tableau and that $\mathfrak{t}$ is a row standard $\mu$-tableau. We say that $\mathfrak{s}$ dominates $\mathfrak{t},$ and we write $\mathfrak{s}\unrhd\mathfrak{t}$ if $\mathfrak{s}_{\downarrow k}\unrhd\mathfrak{t}_{\downarrow k}$ for all $k.$ If $\mathfrak{s}\unrhd\mathfrak{t}$ and $\mathfrak{s}\neq\mathfrak{t},$ then we write $\mathfrak{s}\rhd\mathfrak{t}.$

Let $d\in \mathbb{Z}_{\geq 1}.$ A $d$-composition (resp. $d$-partition) of $n$ is an ordered $d$-tuple $\bm{\lambda}=(\lambda^{(1)},\lambda^{(2)},\ldots,\lambda^{(d)})$ of compositions (resp. partitions) $\lambda^{(k)}$ such that $\sum_{k=1}^{d}|\lambda^{(k)}|=n.$ By [ChPA2, $\S3.1$], the combinatorial objects appearing in the representation theory of the cyclotomic Yokonuma-Hecke algebra $Y_{r,n}^{d}$ will be $(r,d)$-compositions (resp. $(r,d)$-partitions). By definition, an $(r,d)$-composition (resp. $(r,d)$-partition) of $n$ is an ordered $r$-tuple $\underline{\bm{\lambda}}=(\bm{\lambda}^{(1)},\ldots,\bm{\lambda}^{(r)})=((\lambda_{1}^{(1)},\ldots,\lambda_{d}^{(1)}),\ldots,(\lambda_{1}^{(r)},\ldots,\lambda_{d}^{(r)}))$ of $d$-compositions (resp. $d$-partitions) $(\lambda_{1}^{(k)},\ldots,\lambda_{d}^{(k)})$ ($1\leq k\leq r$) such that $\sum_{k=1}^{r}\sum_{j=1}^{d}|\lambda_{j}^{(k)}|=n.$ We denote by $\mathcal{C}_{r,n}^{d}$ (resp. $\mathcal{P}_{r,n}^{d}$) the set of $(r,d)$-compositions (resp. $(r,d)$-partitions) of $n.$ We will say that the $l$-th composition (resp. partition) of the $k$-th $r$-tuple has position $(k,l).$

A triplet $\bm{\theta}=(\theta, k, l)$ consisting of a node $\theta,$ an integer $k\in \{1,\ldots,r\},$ and an integer $l\in \{1,\ldots,d\}$ is called an $(r,d)$-node. We call $k$ the $r$-position of $\bm{\theta}$ and $l$ the $d$-position of $\bm{\theta}.$ We shall also say that the $(r,d)$-node $\bm{\theta}$ has position $(k,l).$ For each $\underline{\bm{\lambda}}\in \mathcal{C}_{r,n}^{d}$ (resp. $\mathcal{P}_{r,n}^{d}$), we shall denote by $[\underline{\bm{\lambda}}]$ the set of $(r,d)$-nodes such that the subset consisting of the $(r,d)$-nodes having position $(k,l)$ forms a usual composition (resp. partition) $\lambda_{l}^{(k)}$, for any $k\in \{1,\ldots,r\}$ and $l\in \{1,\ldots,d\}.$

Let $\underline{\bm{\lambda}}=((\lambda_{1}^{(1)},\ldots,\lambda_{d}^{(1)}),\ldots,(\lambda_{1}^{(r)},\ldots,\lambda_{d}^{(r)}))$ be an $(r,d)$-composition of $n.$ An $(r,d)$-tableau $\mathfrak{t}=((\mathfrak{t}_{1}^{(1)},\ldots,\mathfrak{t}_{d}^{(1)}),\ldots,(\mathfrak{t}_{1}^{(r)},\ldots,\mathfrak{t}_{d}^{(r)}))$ of shape $\underline{\bm{\lambda}}$ is obtained by placing each $(r,d)$-node of $[\underline{\bm{\lambda}}]$ by one of the integers $1,2,\ldots,n,$ allowing no repeats. We will call the number $n$ the size of $\mathfrak{t}$ and the $\mathfrak{t}_{l}^{(k)}$'s the components of $\mathfrak{t}.$ Each $(r,d)$-node $\bm{\theta}$ of $\mathfrak{t}$ is labelled by $((a, b), k, l)$ if it lies in row $a$ and column $b$ of the component $\mathfrak{t}_{l}^{(k)}$ of $\mathfrak{t}.$

For each $\underline{\bm{\mu}}\in \mathcal{C}_{r,n}^{d},$ an $(r,d)$-tableau of shape $\underline{\bm{\mu}}$ is called row standard if the numbers increase along any row (from left to right) of each diagram in $[\underline{\bm{\mu}}].$ For each $\underline{\bm{\lambda}}\in \mathcal{P}_{r,n}^{d},$ an $(r,d)$-tableau of shape $\underline{\bm{\lambda}}$ is called standard if the numbers increase along any row (from left to right) and down any column (from top to bottom) of each diagram in $[\underline{\bm{\lambda}}].$ From now on, we denote by $\text{Std}(\underline{\bm{\lambda}})$ the set of all standard $(r,d)$-tableaux of size $n$ and of shape $\underline{\bm{\lambda}},$ which is endowed with an action of $\mathfrak{S}_{n}$ from the right by permuting the entries in each $(r,d)$-tableau.

For each $\underline{\bm{\lambda}}\in \mathcal{C}_{r,n}^{d},$ we denote by $\mathfrak{t}^{\underline{\bm{\lambda}}}$ the standard $(r,d)$-tableau of shape $\underline{\bm{\lambda}}$ in which $1,2,\ldots,n$ appear in increasing order from left to right along the rows of the first diagram, and then along the rows of the second diagram, and so on.

For each $\underline{\bm{\lambda}}=((\lambda_{1}^{(1)},\ldots,\lambda_{d}^{(1)}),\ldots,(\lambda_{1}^{(r)},\ldots,\lambda_{d}^{(r)}))\in \mathcal{C}_{r,n}^{d},$ we have a Young subgroup
$$\mathfrak{S}_{\underline{\bm{\lambda}}} :=\mathfrak{S}_{\lambda_{1}^{(1)}}\times\cdots\times\mathfrak{S}_{\lambda_{d}^{(1)}}\times\cdots\times
\mathfrak{S}_{\lambda_{1}^{(r)}}\times\cdots\times\mathfrak{S}_{\lambda_{d}^{(r)}},$$
which is exactly the row stabilizer of $\mathfrak{t}^{\underline{\bm{\lambda}}}.$

For each $\underline{\bm{\lambda}}\in \mathcal{C}_{r,n}^{d}$ and a row standard $(r,d)$-tableau $\mathfrak{s}$ of shape $\underline{\bm{\lambda}},$ let $d(\mathfrak{s})$ be the element of $\mathfrak{S}_{n}$ such that $\mathfrak{s}=\mathfrak{t}^{\underline{\bm{\lambda}}}d(\mathfrak{s}).$ Then $d(\mathfrak{s})$ is a distinguished right coset representative of $\mathfrak{S}_{\underline{\bm{\lambda}}}$ in $\mathfrak{S}_{n},$ that is, $l(wd(\mathfrak{s}))=l(w)+l(d(\mathfrak{s}))$ for any $w\in \mathfrak{S}_{\underline{\bm{\lambda}}}.$ In this way, we obtain a correspondence between the set of row standard $(r,d)$-tableaux of shape $\underline{\bm{\lambda}}$ and the set of distinguished right coset representatives of $\mathfrak{S}_{\underline{\bm{\lambda}}}$ in $\mathfrak{S}_{n}.$

\begin{definition}
Let $\underline{\bm{\lambda}}$ be an $(r,d)$-composition of $n$ such that $\mathfrak{t}$ is an $(r,d)$-tableau of shape $\underline{\bm{\lambda}}.$ For each $a=1,2,\ldots,n,$ we write $\text{p}_{\mathfrak{t}}(a)=(\text{p}_{\mathfrak{t}}^{r}(a), \text{p}_{\mathfrak{t}}^{d}(a)) :=(k,l)$ if $a$ appears in the component $\mathfrak{t}_{l}^{(k)}$ of $\mathfrak{t}.$ We say that an $(r,d)$-tableau $\mathfrak{t}$ of shape $\lambda$ is of initial kind if $\text{p}_{\mathfrak{t}}(a)=\text{p}_{\mathfrak{t}^{\underline{\bm{\lambda}}}}(a)$ for all $a=1,2,\ldots,n.$
\end{definition}

We define a partial order on the set of $(r,d)$-compositions and $(r,d)$-tableaux, which is similar to the case of compositions and tableaux.

\begin{definition}\label{Def-def}
Let $\underline{\bm{\lambda}}=((\lambda_{1}^{(1)},\ldots,\lambda_{d}^{(1)}),\ldots,(\lambda_{1}^{(r)},\ldots,\lambda_{d}^{(r)}))$ and $\underline{\bm{\mu}}=((\mu_{1}^{(1)},\ldots,\mu_{d}^{(1)}),$ $\ldots,(\mu_{1}^{(r)},\ldots,\mu_{d}^{(r)}))$ be two $(r,d)$-compositions of $n.$ We say that $\underline{\bm{\lambda}}$ dominates $\underline{\bm{\mu}},$ and we write $\underline{\bm{\lambda}}\unrhd \underline{\bm{\mu}}$ if and only if $$\sum_{i=1}^{k-1}\sum_{j=1}^{d}|\lambda_{j}^{(i)}|+\sum_{j=1}^{l-1}|\lambda_{j}^{(k)}|+\sum_{i=1}^{p}\lambda_{l,i}^{(k)}\geq \sum_{i=1}^{k-1}\sum_{j=1}^{d}|\mu_{j}^{(i)}|+\sum_{j=1}^{l-1}|\mu_{j}^{(k)}|+\sum_{i=1}^{p}\mu_{l,i}^{(k)}$$
for all $k,$ $l$ and $p$ with $1\leq k\leq r,$ $1\leq l\leq d$ and $p\geq 0.$ If $\underline{\bm{\lambda}}\unrhd\underline{\bm{\mu}}$ and $\underline{\bm{\lambda}}\neq \underline{\bm{\mu}},$ we write $\underline{\bm{\lambda}}\rhd \underline{\bm{\mu}}.$
\end{definition}

We extend the partial order above to row standard $(r,d)$-tableaux as follows. If $\mathfrak{v}$ is a row standard $(r,d)$-tableau of shape $\underline{\bm{\lambda}}$ and $1\leq k\leq n,$ then the entries $1,2,\ldots,k$ in $\mathfrak{v}$ occupy the diagrams of an $(r,d)$-composition; let $\mathrm{shape}(\mathfrak{v}_{\downarrow k})$ denote this $(r,d)$-composition. Let $\underline{\bm{\lambda}}, \underline{\bm{\mu}}\in \mathcal{C}_{r,n}^{d}.$ Suppose that $\mathfrak{s}$ is a row standard $(r,d)$-tableau of shape $\underline{\bm{\lambda}}$ and that $\mathfrak{t}$ is a row standard $(r,d)$-tableau of shape $\underline{\bm{\mu}}$. We say that $\mathfrak{s}$ dominates $\mathfrak{t},$ and we write $\mathfrak{s}\unrhd\mathfrak{t}$ if $\mathrm{shape}(\mathfrak{s}_{\downarrow k})\unrhd \mathrm{shape}(\mathfrak{t}_{\downarrow k})$ for all $k.$ If $\mathfrak{s}\unrhd\mathfrak{t}$ and $\mathfrak{s}\neq\mathfrak{t},$ then we write $\mathfrak{s}\rhd\mathfrak{t}.$

\section{Cellular bases of cyclotomic Yokonuma-Hecke algebras}

In this section, largely inspired by the results of [DJM] and [ER], we shall construct an explicit cellular basis of the cyclotomic Yokonuma-Hecke algebra $Y_{r,n}^{d}$.

Let us first recall the definition of a cellular basis following [GL].

\begin{definition}\label{cellular-bases}
Let $k$ be an integral domain. An associative $k$-algebra $A$ is called a cellular algebra with a cell datum $(\Lambda, M, C, i)$ if the following conditions are satisfied:\vskip1.5mm

(C1) The finite set $\Lambda$ is partially ordered. Associated with each $\lambda\in \Lambda$ there is a finite set $M(\lambda).$ The algebra $A$ has a $k$-basis $C_{s, t}^{\lambda},$ where $(s, t)$ runs through all elements of $M(\lambda)\times M(\lambda)$ for all $\lambda\in \Lambda.$\vskip1.5mm

(C2) The map $i$ is a $k$-linear anti-automorphism of $A$ with $i^{2}=id$ which sends $C_{s, t}^{\lambda}$ to $C_{t, s}^{\lambda}.$\vskip1.5mm

(C3) For each $\lambda\in \Lambda,$ $s, t\in M(\lambda)$ and each $a\in A$, the product $C_{s, t}^{\lambda}a$ can be written as $\sum_{u\in M(\lambda)}r_{t}^{u}(a)C_{s, u}^{\lambda}+A^{>\lambda},$ where $r_{t}^{u}(a)\in k$ is independent of $s$ and $A^{>\lambda}$ is the $k$-submodule of $A$ generated by $\{C_{s', t'}^{\mu}\:|\:\mu>\lambda; s', t'\in M(\mu)\}.$
\end{definition}

We now fix once and for all a total order on the set of $r$-th roots of unity via setting $\zeta_{k} :=\zeta^{k-1}$ for $1\leq k\leq r.$ Then we define a set partition $A_{\underline{\bm{\lambda}}}\in \mathcal{SP}_{n}$ for any $(r,d)$-composition $\underline{\bm{\lambda}}.$

\begin{definition}\label{definition-1}
Let $\underline{\bm{\lambda}}=((\lambda_{1}^{(1)},\ldots,\lambda_{d}^{(1)}),\ldots,(\lambda_{1}^{(r)},\ldots,\lambda_{d}^{(r)}))\in \mathcal{C}_{r,n}^{d}.$ Suppose that we choose all $1\leq i_{1}< i_{2}<\cdots < i_{p}\leq r$ such that $(\lambda_{1}^{(i_1)},\ldots,\lambda_{d}^{(i_1)}),$ $(\lambda_{1}^{(i_2)},\ldots,\lambda_{d}^{(i_2)}),\ldots,$$(\lambda_{1}^{(i_p)},$\\$\ldots,\lambda_{d}^{(i_p)})$ are nonempty. Define $a_{k} :=\sum_{j=1}^{k}|\bm{\lambda}^{(i_{j})}|$ for $1\leq k\leq p,$ where $|\bm{\lambda}^{(i_{j})}|=\sum_{l=1}^{d}|\lambda_{l}^{(i_{j})}|.$ Then the set partition $A_{\underline{\bm{\lambda}}}$ associated with $\underline{\bm{\lambda}}$ is defined as $$A_{\underline{\bm{\lambda}}} :=\{\{1,\ldots,a_{1}\},\{a_{1}+1,\ldots,a_{2}\},\ldots,\{a_{p-1}+1,\ldots,n\}\},$$ which may be written as $A_{\underline{\bm{\lambda}}}=\{I_{1},I_{2},\ldots,I_{p}\},$ and is referred to the blocks of $A_{\underline{\bm{\lambda}}}$ in the order given above.
\end{definition}

The following lemma can be easily proved, which we shall use frequently in the sequel.

\begin{lemma}\label{6-3-lemma}
Let $A_{\underline{\bm{\lambda}}}=\{I_{1},I_{2},\ldots,I_{p}\}$ and let $k_{1}, k_{2}\in I_{k}$ for some $1\leq k\leq p.$ Then we have
\begin{equation}\label{tkEAlambda}
t_{k_{1}}E_{A_{\underline{\bm{\lambda}}}}=E_{A_{\underline{\bm{\lambda}}}}t_{k_{1}}=
t_{k_{2}}E_{A_{\underline{\bm{\lambda}}}}=E_{A_{\underline{\bm{\lambda}}}}t_{k_{2}}.
\end{equation}
\end{lemma}

From the basis Theorem \ref{basis-theorem-cyclotomic} we can get that $t_{i}$ acts diagonalizably on $Y_{r,n}^{d}$ with minimal polynomial $t_{i}^{r}-1=\Pi_{l=1}^{r}(t_{i}-\zeta_{l}).$ Hence, if we define $u_{i,k}=\Pi_{l=1;l\neq k}^{r}(t_{i}-\zeta_{l}),$ we get that $u_{i,k}$ is the eigenspace for the action of $t_{i}$ on $Y_{r,n}^{d}$ with eigenvalue $\zeta_{k},$ that is, $$\{y\in Y_{r,n}^{d}\:|\:t_{i}y=\zeta_{k}y\}=u_{i,k}Y_{r,n}^{d}.$$

This motivates us to give the following definition.

\begin{definition}\label{definition-2}
Let $\underline{\bm{\lambda}}=((\lambda_{1}^{(1)},\ldots,\lambda_{d}^{(1)}),\ldots,(\lambda_{1}^{(r)},\ldots,\lambda_{d}^{(r)}))\in \mathcal{C}_{r,n}^{d},$ and let $a_{k} :=\sum_{j=1}^{k}|\bm{\lambda}^{(i_{j})}|$ $(1\leq k\leq p)$ be defined as above. Then we define $$u_{\underline{\bm{\lambda}}} :=u_{a_{1},i_{1}}u_{a_{2},i_{2}}\cdots u_{a_{p},i_{p}}.$$
\end{definition}

From Lemma \ref{6-3-lemma} and the definitions we can easily get the following lemma.

\begin{lemma}\label{6-5-lemma}
We set $U_{\underline{\bm{\lambda}}} :=u_{\underline{\bm{\lambda}}}E_{A_{\underline{\bm{\lambda}}}}.$ Let $A_{\underline{\bm{\lambda}}}=\{I_{1},I_{2},\ldots,I_{p}\}$ and let $k_{d}$ be any element of $I_{d}.$ Then we have $$U_{\underline{\bm{\lambda}}}=\prod_{d=1}^{p}u_{k_{d}, i_{d}}E_{A_{\underline{\bm{\lambda}}}}.$$
In particular, we have for any $i\in I_{d}$ that
\begin{equation}\label{tiU-lambda}
t_{i}U_{\underline{\bm{\lambda}}}=\zeta_{i_{d}}U_{\underline{\bm{\lambda}}}.
\end{equation}
\end{lemma}

\begin{definition}\label{definition-3}
Let $\underline{\bm{\lambda}}=((\lambda_{1}^{(1)},\ldots,\lambda_{d}^{(1)}),\ldots,(\lambda_{1}^{(r)},\ldots,\lambda_{d}^{(r)}))\in \mathcal{C}_{r,n}^{d}.$ Associated with $\underline{\bm{\lambda}}$ we can define the following elements $a_{l}^{k}$ and $b_{k}$:$$a_{l}^{k} :=\sum_{m=1}^{l-1}|\lambda_{m}^{(k)}|,~~~~b_{k} :=\sum_{j=1}^{k-1}\sum_{i=1}^{d}|\lambda_{i}^{(j)}|~~~~\mathrm{for}~1\leq k\leq r~\mathrm{and}~1\leq l\leq d.$$ Associated with these elements we can define an element $u_{\mathbf{a}}^{+} :=u_{\mathbf{a}, 1}u_{\mathbf{a}, 2}\cdots u_{\mathbf{a}, r},$ where $$u_{\mathbf{a}, k} :=\prod_{l=1}^{d}\prod_{j=1}^{a_{l}^{k}}(X_{b_{k}+j}-v_{l}).$$
\end{definition}

We can now define the key ingredient of the cellular basis for $Y_{r,n}^{d}.$

\begin{definition}\label{definition-4}
Let $\underline{\bm{\lambda}}\in \mathcal{C}_{r,n}^{d}$ and define $u_{\mathbf{a}}^{+}$ as above. Let $x_{\underline{\bm{\lambda}}}=\sum_{w\in \mathfrak{S}_{\underline{\bm{\lambda}}}}q^{l(w)}g_{w}.$ Then we define the element $m_{\underline{\bm{\lambda}}}$ of $Y_{r,n}^{d}$ as follows:
\begin{equation}\label{mUlam-lambda}
m_{\underline{\bm{\lambda}}} :=U_{\underline{\bm{\lambda}}}u_{\mathbf{a}}^{+}x_{\underline{\bm{\lambda}}}=
u_{\underline{\bm{\lambda}}}E_{A_{\underline{\bm{\lambda}}}}u_{\mathbf{a}}^{+}x_{\underline{\bm{\lambda}}}.
\end{equation}
\end{definition}

The following lemma gives some basic properties of the element $m_{\underline{\bm{\lambda}}}.$
\begin{lemma}\label{lemma6-8-Ueij}
$(a)$ suppose that $\underline{\bm{\lambda}}=((\lambda_{1}^{(1)},\ldots,\lambda_{d}^{(1)}),\ldots,(\lambda_{1}^{(r)},\ldots,\lambda_{d}^{(r)}))\in \mathcal{C}_{r,n}^{d},$ and let $\alpha=(|\lambda_{1}^{(1)}|,\ldots,|\lambda_{d}^{(1)}|,\ldots,|\lambda_{1}^{(r)}|,\ldots,|\lambda_{d}^{(r)}|).$ If $w\in \mathfrak{S}_{\alpha},$ then $g_{w}$ commutes with $U_{\underline{\bm{\lambda}}}$ and $u_{\mathbf{a}}^{+},$ respectively. In particular, we have that
\begin{equation}\label{mueaxu-lambda}
m_{\underline{\bm{\lambda}}}=u_{\underline{\bm{\lambda}}}E_{A_{\underline{\bm{\lambda}}}}x_{\underline{\bm{\lambda}}}u_{\mathbf{a}}^{+}=
u_{\underline{\bm{\lambda}}}x_{\underline{\bm{\lambda}}}E_{A_{\underline{\bm{\lambda}}}}u_{\mathbf{a}}^{+}
=x_{\underline{\bm{\lambda}}}u_{\underline{\bm{\lambda}}}E_{A_{\underline{\bm{\lambda}}}}u_{\mathbf{a}}^{+}=
x_{\underline{\bm{\lambda}}}u_{\mathbf{a}}^{+}u_{\underline{\bm{\lambda}}}E_{A_{\underline{\bm{\lambda}}}}.
\end{equation}

$(b)$ If $e_{i,j}$ does not appear in the product $E_{A_{\underline{\bm{\lambda}}}}$, that is, if $i$ and $j$ occur in two different blocks of $A_{\underline{\bm{\lambda}}},$ then we have $U_{A_{\underline{\bm{\lambda}}}}e_{i,j}=0.$
\end{lemma}
\begin{proof}
$(a)$ Lemma \ref{5-1-lemmaaa} implies that $g_{w}$ and $E_{A_{\underline{\bm{\lambda}}}}$ commute for $w\in \mathfrak{S}_{\alpha}.$ Moreover, from the definition of $u_{i,k}$ we get that $u_{i,k}g_{w}=g_{w}u_{iw, k},$ and so Lemma \ref{6-5-lemma} implies that $g_{w}$ commutes with $u_{\underline{\bm{\lambda}}}$ for $w\in \mathfrak{S}_{\alpha}.$

$(b)$ It follows from Lemmas \ref{6-3-lemma} and \ref{6-5-lemma}.
\end{proof}

\begin{lemma}\label{lemma6-9-mgw}
Assume that $\underline{\bm{\lambda}}\in \mathcal{C}_{r,n}^{d}$. Then we have the following equation$:$
\begin{equation}\label{mlambdagw-qwlam}
m_{\underline{\bm{\lambda}}}g_{w}=q^{l(w)}m_{\underline{\bm{\lambda}}}~~~~~for~all~w\in \mathfrak{S}_{\underline{\bm{\lambda}}}.
\end{equation}
\end{lemma}

\begin{proof}
For any $i$ such that $s_{i}\in \mathfrak{S}_{\underline{\bm{\lambda}}},$ we have $$m_{\underline{\bm{\lambda}}}g_{i}=\sum_{\substack{w\in \mathfrak{S}_{\underline{\bm{\lambda}}}\\l(ws_{i})>l(w)}}q^{l(w)}u_{\underline{\bm{\lambda}}}E_{A_{\underline{\bm{\lambda}}}}g_{w}g_{i}+\sum_{\substack{w\in \mathfrak{S}_{\underline{\bm{\lambda}}}\\l(ws_{i})<l(w)}}q^{l(w)}u_{\underline{\bm{\lambda}}}E_{A_{\underline{\bm{\lambda}}}}g_{ws_{i}}g_{i}^{2}.$$

Since $ws_{i}\in \mathfrak{S}_{\underline{\bm{\lambda}}},$ we have $E_{A_{\underline{\bm{\lambda}}}}g_{ws_{i}}=g_{ws_{i}}E_{A_{\underline{\bm{\lambda}}}}$ by Lemma \ref{5-1-lemmaaa}. While $E_{A_{\underline{\bm{\lambda}}}}g_{i}^{2}=E_{A_{\underline{\bm{\lambda}}}}(1+(q-q^{-1})g_{i}),$ thus we have
\begin{align*}
m_{\underline{\bm{\lambda}}}g_{i}=&\sum_{\substack{w\in \mathfrak{S}_{\underline{\bm{\lambda}}}\\l(ws_{i})<l(w)}}q^{l(w)-1}u_{\underline{\bm{\lambda}}}E_{A_{\underline{\bm{\lambda}}}}g_{w}
+\sum_{\substack{w\in \mathfrak{S}_{\underline{\bm{\lambda}}}\\l(ws_{i})<l(w)}}q^{l(w)}u_{\underline{\bm{\lambda}}}E_{A_{\underline{\bm{\lambda}}}}g_{ws_{i}}\\
&+\sum_{\substack{w\in \mathfrak{S}_{\underline{\bm{\lambda}}}\\l(ws_{i})<l(w)}}q^{l(w)}(q-q^{-1})u_{\underline{\bm{\lambda}}}E_{A_{\underline{\bm{\lambda}}}}g_{w}\\
=&qm_{\underline{\bm{\lambda}}}.
\end{align*}

We are done.
\end{proof}

In this section, let us rewrite the anti-automorphism $\tau$ on $Y_{r,n}^{d}$ introduced in Section 4 by $\ast,$ which is determined by $$g_{i}^{\ast}=g_{i},\quad t_{j}^{\ast}=t_{j},\quad X_{j}^{\ast}=X_{j}\quad \mathrm{for}~1\leq i\leq n-1~\mathrm{and}~1\leq j\leq n.$$

\begin{definition}
Let $\underline{\bm{\lambda}}\in \mathcal{C}_{r,n}^{d},$ and let $\mathfrak{s}$ and $\mathfrak{t}$ be two row standard $(r,d)$-tableaux of shape $\underline{\bm{\lambda}}.$ We then define $m_{\mathfrak{s}\mathfrak{t}}=g_{d(\mathfrak{s})}^{\ast}m_{\underline{\bm{\lambda}}}g_{d(\mathfrak{t})}.$
\end{definition}

Since $m_{\underline{\bm{\lambda}}}^{\ast}=m_{\underline{\bm{\lambda}}},$ we have $m_{\mathfrak{s}\mathfrak{t}}^{\ast}=m_{\mathfrak{t}\mathfrak{s}}.$

One of the aims of this section is to show that the elements $m_{\mathfrak{s}\mathfrak{t}},$ as $(\mathfrak{s}, \mathfrak{t})$ runs over the ordered pairs of standard $(r,d)$-tableaux of the same shape, give a cellular basis of $Y_{r,n}^{d}.$

\begin{lemma}\label{lemma6-11-mlam}
Suppose that $\underline{\bm{\lambda}}\in \mathcal{C}_{r,n}^{d}$ and that $\mathfrak{s}$ and $\mathfrak{t}$ are row standard $(r,d)$-tableaux of shape $\underline{\bm{\lambda}}.$ For each $h\in Y_{r,n},$ we have that $m_{\mathfrak{s}\mathfrak{t}}h$ is a linear combination of terms of the form $m_{\mathfrak{s}\mathfrak{v}},$ where $\mathfrak{v}$ is a row standard $(r,d)$-tableau of shape $\underline{\bm{\lambda}}.$
\end{lemma}
\begin{proof}
The proof is similar to that of [ER, Lemma 16] by using the basis Theorem \ref{basis-theorem-cyclotomic} and the fact that $m_{\underline{\bm{\lambda}}}t_{i}=\zeta_{i_{d}}m_{\underline{\bm{\lambda}}}$ for any $i\in I_{d}$. We skip the details.
\end{proof}

The proof of the next lemma is similar to that of [DJM, Lemma 3.17].

\begin{lemma}\label{lemma6-12-mlambda}
Suppose that $\underline{\bm{\lambda}}=((\lambda_{1}^{(1)},\ldots,\lambda_{d}^{(1)}),\ldots,(\lambda_{1}^{(r)},\ldots,\lambda_{d}^{(r)}))\in \mathcal{C}_{r,n}^{d},$ and let $\alpha=(|\lambda_{1}^{(1)}|,\ldots,|\lambda_{d}^{(1)}|,\ldots,|\lambda_{1}^{(r)}|,\ldots,|\lambda_{d}^{(r)}|).$ Suppose that $w$ is a distinguished right coset representative of $\mathfrak{S}_{\alpha}$ in $\mathfrak{S}_{n}$ and that $\mathfrak{s}$ is a row standard $(r,d)$-tableau of initial kind. Then the following hold$:$

$(a)$ The tableau $\mathfrak{s}w$ is row standard. Moreover, if $\mathfrak{s}$ is standard then $\mathfrak{s}w$ is standard.

$(b)$ If $\mathfrak{t}$ is a row standard $(r,d)$-tableau of initial kind with $\mathfrak{s}\rhd\mathfrak{t}$ then $\mathfrak{s}w\rhd\mathfrak{t}w.$
\end{lemma}

The proof of the following proposition is inspired by the proof of [DJM, Proposition 3.18] and [ER, Lemma 17]. It allows us to restrict ourselves to the case of $(r,d)$-partitions.

\begin{proposition}\label{propo6-13}
Suppose that $\underline{\bm{\lambda}}\in \mathcal{C}_{r,n}^{d}$ and that $\mathfrak{s}$ and $\mathfrak{t}$ are row standard $(r,d)$-tableaux of shape $\underline{\bm{\lambda}}.$ Then $m_{\mathfrak{s}\mathfrak{t}}$ is a linear combination of terms of the form $m_{\mathfrak{u}\mathfrak{v}},$ where $\mathfrak{u}$ and $\mathfrak{v}$ are standard $(r,d)$-tableaux of shape $\underline{\bm{\mu}}$ such that $\mathfrak{u}\unrhd \mathfrak{s},$ $\mathfrak{v}\unrhd \mathfrak{t}$ and $\underline{\bm{\mu}}\in \mathcal{P}_{r,n}^{d}.$
\end{proposition}
\begin{proof}
When $r=d=1,$ we can adapt the proof of [Mu, Theorem 4.18] (see also [Ma1, Lemma 3.14]) to our setting by using Lemma \ref{lemma6-9-mgw} among other things. Thus we can conclude that it is true in this case.

In the general case, suppose that $\underline{\bm{\lambda}}=((\lambda_{1}^{(1)},\ldots,\lambda_{d}^{(1)}),\ldots,(\lambda_{1}^{(r)},\ldots,\lambda_{d}^{(r)}))\in \mathcal{C}_{r,n}^{d},$ and let $\alpha=(|\lambda_{1}^{(1)}|,\ldots,|\lambda_{d}^{(1)}|,\ldots,|\lambda_{1}^{(r)}|,\ldots,|\lambda_{d}^{(r)}|).$ We may write $\mathfrak{s}=\mathfrak{s}'w_{1}$ and $\mathfrak{t}=\mathfrak{t}'w_{2},$ where $\mathfrak{s}'$ and $\mathfrak{t}'$ are row standard $(r, d)$-tableaux of shape $\underline{\bm{\lambda}}$ of initial kind, and $w_{1}$ and $w_{2}$ are distinguished right coset representatives for $\mathfrak{S}_{\alpha}$ in $\mathfrak{S}_{n}.$ Then we have $d(\mathfrak{s})=d(\mathfrak{s}')w_{1},$ $d(\mathfrak{t})=d(\mathfrak{t}')w_{2}$ and $l(d(\mathfrak{s}))=l(d(\mathfrak{s}'))+l(w_{1}),$ $l(d(\mathfrak{t}))=l(d(\mathfrak{t}'))+l(w_{2}).$ Therefore, we have $m_{\mathfrak{s}\mathfrak{t}}=g_{w_{1}}^{\ast}m_{\mathfrak{s}'\mathfrak{t}'}g_{w_{2}},$ and $$m_{\mathfrak{s}'\mathfrak{t}'}=g_{d(\mathfrak{s}')}^{\ast}u_{\underline{\bm{\lambda}}}E_{A_{\underline{\bm{\lambda}}}}
u_{\mathbf{a}}^{+}x_{\underline{\bm{\lambda}}}g_{d(\mathfrak{t}')}=
u_{\mathbf{a}}^{+}U_{\underline{\bm{\lambda}}}g_{d(\mathfrak{s}')}^{\ast}x_{\underline{\bm{\lambda}}}g_{d(\mathfrak{t}')}.$$

We may write $g_{d(\mathfrak{s}')}^{\ast}x_{\underline{\bm{\lambda}}}g_{d(\mathfrak{t}')}$ as a product of $rd$ commuting terms, one for each component of $\underline{\bm{\lambda}};$ that is, $$g_{d(\mathfrak{s}')}^{\ast}x_{\underline{\bm{\lambda}}}g_{d(\mathfrak{t}')}=x_{1}^{(1)}\cdots x_{d}^{(1)}\cdots x_{1}^{(r)}\cdots x_{d}^{(r)},$$ where the $(k,l)$-th term $x_{l}^{(k)}$ involves only elements $g_{w}$ with $w\in \mathfrak{S}(\{b_{k}+a_{l}^{k}+1,b_{k}+a_{l}^{k}+2,\ldots,b_{k}+a_{l+1}^{k}\}).$ For example, $x_{1}^{(1)}=g_{d(\mathfrak{s}'^{(1)}_{1})}^{\ast}x_{\lambda_{1}^{(1)}}g_{d(\mathfrak{t}'^{(1)}_{1})},$ where $\mathfrak{s}'^{(1)}_{1}$ (resp. $\mathfrak{t}'^{(1)}_{1}$) is the first component of $\mathfrak{s}'$ (resp. $\mathfrak{t}'$).

By applying the results of the special case $r=d=1,$ we may write each $U_{\underline{\bm{\lambda}}}x_{l}^{(k)}$ as a linear combination of terms $U_{\underline{\bm{\lambda}}}g_{d(\mathfrak{u}'^{(k)}_{l})}^{\ast}x_{\lambda_{l}^{(k)}}g_{d(\mathfrak{v}'^{(k)}_{l})},$ where $\mathfrak{u}'^{(k)}_{l}$ and $\mathfrak{v}'^{(k)}_{l}$ are standard $\mu^{(k)}_{l}$-tableau for some partition $\mu_{l}^{(k)}$ such that $|\mu_{l}^{(k)}|=|\lambda_{l}^{(k)}|,$ and satisfy $\mathfrak{u}'^{(k)}_{l}\unrhd \mathfrak{s}'^{(k)}_{l}$ and $\mathfrak{v}'^{(k)}_{l}\unrhd \mathfrak{t}'^{(k)}_{l}.$ We then conclude that $U_{\underline{\bm{\lambda}}}g_{d(\mathfrak{s}')}^{\ast}x_{\underline{\bm{\lambda}}}g_{d(\mathfrak{t}')}$ is a linear combination of terms of the form $U_{\underline{\bm{\lambda}}}g_{d(\mathfrak{u}')}^{\ast}x_{\underline{\bm{\mu}}}g_{d(\mathfrak{v}')},$ where $\mathfrak{u}'$ and $\mathfrak{v}'$ are standard $(r,d)$-tableaux of shape $\underline{\bm{\mu}}$ for some $(r,d)$-partition $\underline{\bm{\mu}}$ of $n,$ and moreover, $\mathfrak{u}'\unrhd \mathfrak{s}'$ and $\mathfrak{v}'\unrhd \mathfrak{t}'.$ We also have that $\mathfrak{u}'$ and $\mathfrak{s}'$ are of initial kind, and $|\mu_{l}^{(k)}|=|\lambda_{l}^{(k)}|$ for all $1\leq k\leq r$ and $1\leq l\leq d.$ Therefore, $m_{\mathfrak{s}'\mathfrak{t}'}$ is a linear combination of elements $$u_{\mathbf{a}}^{+}U_{\underline{\bm{\mu}}}g_{d(\mathfrak{u}')}^{\ast}x_{\underline{\bm{\mu}}}g_{d(\mathfrak{v}')}=
g_{d(\mathfrak{u}')}^{\ast}m_{\underline{\bm{\mu}}}g_{d(\mathfrak{v}')}=
m_{\mathfrak{u}'\mathfrak{v}'},$$ where the sum runs over the same set of pairs $(\mathfrak{u}', \mathfrak{v}')$ as above. Thus, $m_{\mathfrak{s}\mathfrak{t}}=g_{w_{1}}^{\ast}m_{\mathfrak{s}'\mathfrak{t}'}g_{w_{2}}$ is a linear combination of terms $g_{w_{1}}^{\ast}m_{\mathfrak{u}'\mathfrak{v}'}g_{w_{2}}.$

Since $\mathfrak{u}'$ (resp. $\mathfrak{s}'$) is of initial kind and $w_{1}$ (resp. $w_{2}$) is a distinguished right coset representative for $\mathfrak{S}_{\alpha}$ in $\mathfrak{S}_{n},$ we have, by Lemma \ref{lemma6-12-mlambda}, $\mathfrak{u}=\mathfrak{u}'w_{1}$ (resp. $\mathfrak{v}=\mathfrak{v}'w_{2}$) is standard, and $\mathfrak{u}'w_{1}\unrhd \mathfrak{s}'w_{1}$ (resp. $\mathfrak{v}'w_{2}\unrhd \mathfrak{t}'w_{2}$); that is, $\mathfrak{u}\unrhd \mathfrak{s}$ (resp. $\mathfrak{v}\unrhd \mathfrak{t}$). So we have $g_{w_{1}}^{\ast}m_{\mathfrak{u}'\mathfrak{v}'}g_{w_{2}}=m_{\mathfrak{u}\mathfrak{v}}.$ We have proved the proposition.
\end{proof}

Combining Lemma \ref{lemma6-11-mlam} and Proposition \ref{propo6-13}, we immediately get the following result.

\begin{corollary}\label{corollary6-14}
Suppose that $\underline{\bm{\lambda}}\in \mathcal{C}_{r,n}^{d}$ and that $\mathfrak{s}$ and $\mathfrak{t}$ are row standard $(r,d)$-tableaux of shape $\underline{\bm{\lambda}}.$ If $h\in Y_{r,n}$, then $m_{\mathfrak{s}\mathfrak{t}}h$ is a linear combination of terms of the form $m_{\mathfrak{u}\mathfrak{v}},$ where $\mathfrak{u}$ and $\mathfrak{v}$ are standard $(r,d)$-tableaux of shape $\underline{\bm{\mu}}$ for some $(r,d)$-partition $\underline{\bm{\mu}}$ such that $\mathfrak{u}\unrhd \mathfrak{s}$ and $\underline{\bm{\mu}}\unrhd \underline{\bm{\lambda}}.$
\end{corollary}

The proof of the following proposition is inspired by the proof of [DJM, Proposition 3.20], which shows what happens when we multiply $m_{\mathfrak{s}\mathfrak{t}}$ by $X_{1}$.

\begin{proposition}\label{proposition6-15X}
Suppose that $\underline{\bm{\lambda}}=((\lambda_{1}^{(1)},\ldots,\lambda_{d}^{(1)}),\ldots,(\lambda_{1}^{(r)},\ldots,\lambda_{d}^{(r)}))\in \mathcal{P}_{r,n}^{d}$ and that $\mathfrak{s}$ and $\mathfrak{t}$ are standard $(r,d)$-tableaux of shape $\underline{\bm{\lambda}}.$ Then $m_{\mathfrak{s}\mathfrak{t}}X_{1}=x_{1}+x_{2},$ where

$(1)$ $x_{1}$ is a linear combination of terms of the form $m_{\mathfrak{u}\mathfrak{v}},$ where $\mathfrak{u}$ and $\mathfrak{v}$ are standard

\qquad \hspace{-2mm}$(r,d)$-tableaux of shape $\underline{\bm{\lambda}}$ with $\mathfrak{u}\unrhd \mathfrak{s},$ and

$(2)$ $x_{2}$ is a linear combination of terms of the form $m_{\mathfrak{u}\mathfrak{v}},$ where $\mathfrak{u}$ and $\mathfrak{v}$ are standard

\qquad \hspace{-2mm}$(r,d)$-tableaux of shape $\underline{\bm{\mu}}$ for some $(r,d)$-partition $\mu$ with $\underline{\bm{\mu}}\rhd \underline{\bm{\lambda}}.$
\end{proposition}

\begin{proof}
Let $\alpha=((|\lambda_{1}^{(1)}|,\ldots,|\lambda_{d}^{(1)}|),\ldots,(|\lambda_{1}^{(r)}|,\ldots,|\lambda_{d}^{(r)}|)),$ and let $c=(c_{1}^{1},\ldots,c_{d}^{1},\ldots,c_{1}^{r},\\\ldots,c_{d}^{r}),$ where $$c_{l}^{k}=a_{l}^{k}+b_{k}=\sum_{m=1}^{l-1}|\lambda_{m}^{(k)}|+\sum_{j=1}^{k-1}\sum_{i=1}^{d}|\lambda_{i}^{(j)}|~~~~\mathrm{for}~1\leq k\leq r~\mathrm{and}~1\leq l\leq d.$$

We write $d(\mathfrak{t})=yc$ with $y\in \mathfrak{S}_{\alpha}$ and $c$ a distinguished right coset representative for $\mathfrak{S}_{\alpha}$ in $\mathfrak{S}_{n}.$ Then the $\alpha$-tableau $\mathfrak{t}^{\alpha}c$ is row standard. Assume that $1$ is in row $j$ of the $k$-th component of $\mathfrak{t}^{\alpha}c$ with $1\leq k\leq r$ and $1\leq j\leq d.$

Let $c=(c_{j}^{k}, c_{j}^{k}+1)(c_{j}^{k}-1, c_{j}^{k})\cdots (1,2)c',$ where $l(c)=c_{j}^{k}+l(c')$ and $c'$ fixes 1. Thus, $g_{c}=g_{c_{j}^{k}}g_{c_{j}^{k}-1}\cdots g_{1}g_{c'}$ and $g_{c'}X_{1}=X_{1}g_{c'}$. Let $b=(c_{1}^{1},\ldots,c_{d}^{1},\ldots,c_{1}^{k},\ldots,c_{j-1}^{k}, c_{j}^{k}+1, c_{j+1}^{k},\ldots,c_{d}^{k},\ldots,c_{1}^{r},\ldots,c_{d}^{r}).$ Then we have $$u_{\mathbf{a}}^{+}g_{c}X_{1}=u_{\mathbf{a}}^{+}g_{c_{j}^{k}}g_{c_{j}^{k}-1}\cdots g_{1}X_{1}g_{c'}=u_{\mathbf{a}}^{+}h_{1}+u_{\mathbf{b}}^{+}h_{2}$$ for some $h_{1}, h_{2}\in Y_{r,n}$ by the same argument as in [DJM, Lemma 3.4]. Since $y\in \mathfrak{S}_{\alpha}$ and $\mathfrak{t}^{\lambda}y$ is standard, $y$ fixes $c_{j}^{k}+1;$ therefore, $g_{y}$ commutes with $u_{\mathbf{b}}^{+}.$ Hence we have
\begin{align*}
m_{\mathfrak{s}\mathfrak{t}}X_{1}&=g_{d(\mathfrak{s})}^{\ast}u_{\underline{\bm{\lambda}}}E_{A_{\underline{\bm{\lambda}}}}u_{\mathbf{a}}^{+}x_{\underline{\bm{\lambda}}}g_{d(\mathfrak{t})}\\
&=g_{d(\mathfrak{s})}^{\ast}u_{\underline{\bm{\lambda}}}E_{A_{\underline{\bm{\lambda}}}}x_{\underline{\bm{\lambda}}}u_{\mathbf{a}}^{+}g_{y}g_{c}X_{1}\\
&=g_{d(\mathfrak{s})}^{\ast}u_{\underline{\bm{\lambda}}}E_{A_{\underline{\bm{\lambda}}}}x_{\underline{\bm{\lambda}}}g_{y}u_{\mathbf{a}}^{+}g_{c}X_{1}\\
&=g_{d(\mathfrak{s})}^{\ast}u_{\underline{\bm{\lambda}}}E_{A_{\underline{\bm{\lambda}}}}x_{\underline{\bm{\lambda}}}g_{y}(u_{\mathbf{a}}^{+}h_{1}+u_{\mathbf{b}}^{+}h_{2})\\
&=g_{d(\mathfrak{s})}^{\ast}u_{\underline{\bm{\lambda}}}E_{A_{\underline{\bm{\lambda}}}}x_{\underline{\bm{\lambda}}}(u_{\mathbf{a}}^{+}g_{y}h_{1}+u_{\mathbf{b}}^{+}g_{y}h_{2}).
\end{align*}

The first term, $g_{d(\mathfrak{s})}^{\ast}u_{\underline{\bm{\lambda}}}E_{A_{\underline{\bm{\lambda}}}}x_{\underline{\bm{\lambda}}}u_{\mathbf{a}}^{+}g_{y}h_{1},$ is a linear combination of terms of the required form by Corollary \ref{corollary6-14}. If $j=1,$ then $u_{\mathbf{b}}^{+}=0.$ Assume that $j\geq 2.$ We turn our attention to the second term $g_{d(\mathfrak{s})}^{\ast}u_{\underline{\bm{\lambda}}}E_{A_{\underline{\bm{\lambda}}}}x_{\underline{\bm{\lambda}}}u_{\mathbf{b}}^{+}g_{y}h_{2}$.

We define an $(r, d)$-composition $\underline{\bm{\nu}}=((\nu_{1}^{(1)},\ldots,\nu_{d}^{(1)}),\ldots,(\nu_{1}^{(r)},\ldots,\nu_{d}^{(r)}))$ of $n$ by $\nu_{q}^{(p)}=\lambda_{q}^{(p)}$ for $p\neq k$ and $1\leq q\leq d,$ $\nu_{i}^{(k)}=\lambda_{i}^{(k)}$ for $i\neq j-1, j,$ $\nu_{j-1}^{(k)}=(\lambda_{j-1, 1}^{(k)},\lambda_{j-1, 2}^{(k)},\ldots,1)$ and $\nu_{j}^{(k)}=(\lambda_{j, 1}^{(k)}-1,\lambda_{j, 2}^{(k)},\ldots).$ Note that $\underline{\bm{\nu}}\rhd \underline{\bm{\lambda}}.$

Let $l=\lambda_{j, 1}^{(k)}.$ Let $G_{1}$ be the symmetric group on these numbers $c_{j}^{k}+1,\ldots,c_{j}^{k}+1,$ and let $G_{2}$ be the symmetric group on these numbers $c_{j}^{k}+2,\ldots,c_{j}^{k}+1.$ Let $c_{1},c_{2},\ldots,c_{l}$ be the distinguished right coset representatives for $G_{2}$ in $G_{1},$ ordered in terms of increasing order. Let $\mathfrak{u}_{i}=\mathfrak{t}^{\nu}c_{i}d(\mathfrak{s})$ for $1\leq i\leq l.$

Then by the same argument as in the proof of [DJM, Proposition 3.20], we can get that
\begin{align*}
g_{d(\mathfrak{s})}^{\ast}u_{\lambda}E_{A_{\lambda}}x_{\lambda}u_{\mathbf{b}}^{+}&=
g_{d(\mathfrak{s})}^{\ast}\sum_{i=1}^{l}g_{c_{i}}^{\ast}x_{\nu}u_{\lambda}E_{A_{\lambda}}u_{\mathbf{b}}^{+}\\
&=\sum_{i=1}^{l}g_{c_{i}d(\mathfrak{s})}^{\ast}x_{\nu}u_{\nu}E_{A_{\nu}}u_{\mathbf{b}}^{+}\\
&=\sum_{i=1}^{l}m_{\mathfrak{u}_{i}\mathfrak{t}^{\nu}}.
\end{align*}

Hence, $g_{d(\mathfrak{s})}^{\ast}u_{\underline{\bm{\lambda}}}E_{A_{\underline{\bm{\lambda}}}}x_{\underline{\bm{\lambda}}}u_{\mathbf{b}}^{+}g_{y}h_{2}$ is a linear combination of terms of the required form by Corollary \ref{corollary6-14}. We have completed the proof of this proposition.
\end{proof}

The proof of the next lemma is similar to that of [ER, Lemma 15] by using Lemma \ref{lemma6-8-Ueij}(b).

\begin{lemma}\label{lemma6-16-lambd}
Suppose that $\underline{\bm{\lambda}}\in \mathcal{C}_{r,n}^{d}$ and that $\mathfrak{s}$ and $\mathfrak{t}$ are standard $(r,d)$-tableaux of shape $\underline{\bm{\lambda}}.$ If $i$ $($resp. $j$$)$ occurs in the component $\mathfrak{t}_{l}^{(k)}$ $($resp. $\mathfrak{t}_{l'}^{(k')}$$)$ of $\mathfrak{t}$ with $k\neq k'$, then we have $m_{\mathfrak{s}\mathfrak{t}}e_{i,j}=0.$
\end{lemma}

Let us denote by $\mathcal{T}_{n}$ the $\mathcal{R}$-subalgebra of $Y_{r,n}^{d}$ generated by $t_{1},t_{2},\ldots,t_{n}.$ Using Lemma \ref{lemma6-16-lambd}, we can prove the following proposition by the same argument as in [ER, Lemma 18 and Proposition 1].

\begin{proposition}\label{proposi6-17-mss}
We set $$\mathcal{M}_{n} :=\Big\{\mathfrak{s}\:|\:\mathfrak{s}\in \mathrm{Std}(((0),\ldots,(1^{n_{1}})),\ldots,((0),\ldots,(1^{n_{r}}))),~where~n_{i}\geq 0~and~\sum_{1\leq i\leq r}n_{i}=n\Big\}.$$
Then for all $\mathfrak{s}\in \mathcal{M}_{n},$ we have $m_{\mathfrak{s}\mathfrak{s}}$ belongs to $\mathcal{T}_{n}.$ Moreover, the set $\{m_{\mathfrak{s}\mathfrak{s}}\:|\:\mathfrak{s}\in \mathcal{M}_{n}\}$ is an $\mathcal{R}$-basis of $\mathcal{T}_{n}.$
\end{proposition}

Now we can state the main result of this section, which is in fact a generalization of [ER, Theorem 20].
\begin{theorem}\label{cellular-bases-6-18}
The cyclotomic Yokonuma-Hecke algebra $Y_{r,n}^{d}$ is a free $\mathcal{R}$-module with basis $$\mathcal{B}_{r,n}^{d}=\{m_{\mathfrak{s}\mathfrak{t}}\:|\:\mathfrak{s}, \mathfrak{t}\in \mathrm{Std}(\underline{\bm{\lambda}})~for~some~(r,d)\mathrm{-}partition~\underline{\bm{\lambda}}~of~n\}.$$
Moreover, $(Y_{r,n}^{d}, \mathcal{B}_{r,n}^{d})$ is a cellular basis of $Y_{r,n}^{d}$.
\end{theorem}
\begin{proof}
By Proposition \ref{proposi6-17-mss} we have that 1 is an $\mathcal{R}$-linear combination of elements $m_{\mathfrak{s}\mathfrak{s}}$ for some standard $(r,d)$-tableaux $\mathfrak{s}.$ Hence, using Corollary \ref{corollary6-14} and Proposition \ref{proposition6-15X}, we get that $\mathcal{B}_{r,n}^{d}$ spans $Y_{r,n}^{d}$. By the basis Theorem \ref{basis-theorem-cyclotomic}, $Y_{r,n}^{d}$ is free of rank $(rd)^{n}n!.$ On the other hand, it is easy to see that the cardinality of $\mathcal{B}_{r,n}^{d}$ is $(rd)^{n}n!.$ Thus, we must have that $\mathcal{B}_{r,n}^{d}$ is an $\mathcal{R}$-basis of $Y_{r,n}^{d}$.

Finally, the multiplicative property that $\mathcal{B}_{r,n}^{d}$ must satisfy in order to be a cellular basis of $Y_{r,n}^{d}$ can be deduced by the same argument as in [DJM, Proposition 3.25]. We skip the details.
\end{proof}

\section{Jucys-Murphy elements for cyclotomic Yokonuma-Hecke algebras}

In this section, inspired by the results of [ER, Section 5], we shall show that the Jucys-Murphy elements $X_{i}$ and $t_{i}$ ($1\leq i\leq n$) for $Y_{r,n}^{d}$, with respect to the cellular basis of $Y_{r,n}^{d}$ obtained in Theorem \ref{cellular-bases-6-18}, are JM-elements in the abstract sense defined by Mathas (see [Ma2]). For the split semisimple $Y_{r,n}^{d}$, we then apply the general theory developed in [Ma2, Section 3] to define the idempotents $E_{\mathfrak{t}}$ of $Y_{r,n}^{d}$ and deduce some properties of them.

\subsection{Jucys-Murphy elements}
\begin{definition}\label{definition-7-1}
Let $\mathfrak{t}$ be an $(r,d)$-tableau of shape $\underline{\bm{\lambda}}$ and suppose that the $(r,d)$-node $\bm{\theta}$ of $\mathfrak{t}$ labelled by $((a, b), k, l)$ is filled in with the element $i$ $(1\leq i\leq n)$. Then we define the content of $i$ as the element $\text{c}_{\mathfrak{t}}(i) :=v_{l}q^{2(b-a)}$ and the $r$-position of $i$ as the element $\text{p}_{\mathfrak{t}}^{r}(i) :=k,$ respectively.
\end{definition}

\begin{definition}\label{definition-7-2-jucys-murphy}
Let $k$ be an integral domain. Given a $k$-algebra $A$ with a cellular basis $\mathcal{C}=\{C_{s, t}^{\lambda}\:|\:\lambda\in \Lambda, s, t\in M(\lambda) \},$ where each set $M(\lambda)$ is endowed with a poset structure with ordering $\rhd_{\lambda},$ we say that a set $\mathcal{L}=\{L_{1},\ldots,L_{n}\}\subseteq A$ is a family of JM-elements for $A$ with respect to the basis $\mathcal{C}$ if it satisfies the following conditions:

$(1)$ $\mathcal{L}$ is a set of commuting elements and $L_{i}^{*}=L_{i}$ for each $1\leq i\leq n$.

$(2)$ There exists a set of scalars $\{C_{t}(i)\:|\:t\in M(\lambda), 1\leq i\leq n\}$ such that for all $\lambda\in \Lambda$ and $s, t\in M(\lambda)$ we have that
\begin{equation}\label{defini7-2-identity}
C_{s, t}^{\lambda}L_{i}\equiv C_{t}(i)C_{s, t}^{\lambda}+\sum\limits_{\substack{v\in M(\lambda)\\v\rhd_{\lambda}t}}R_{vt}C_{s, v}^{\lambda}\quad \mathrm{mod}~A^{>\lambda}
\end{equation}
for some $R_{vt}\in k$ and $1\leq i\leq n$.
\end{definition}

The next proposition claims that the following set
$$\mathcal{L}_{Y_{r,n}^{d}}=\{L_{1},\ldots,L_{2n}\:|\:L_{k}=X_{k}, L_{n+k}=t_{k}~\mathrm{for}~1\leq k\leq n\}$$
is a family of JM-elements for $Y_{r,n}^{d}$ with respect to the cellular basis defined in Theorem \ref{cellular-bases-6-18}.

\begin{proposition}\label{proposition7-3-jm}
$(Y_{r,n}^{d}, \mathcal{B}_{r,n}^{d})$ is a cellular algebra with a family of JM-elements $\mathcal{L}_{Y_{r,n}^{d}}.$
\end{proposition}
\begin{proof}
We need to verify that the elements of $\mathcal{L}_{Y_{r,n}^{d}}$ satisfy the conditions in Definition \ref{definition-7-2-jucys-murphy}. By \eqref{commutation-formulae}, $\mathcal{L}_{Y_{r,n}^{d}}$ consists of a set of commutative elements.

We have proved that $\mathcal{B}_{r,n}^{d}$ is a cellular basis of $Y_{r,n}^{d}$ with respect to the dominance order $\unrhd$ on Std$(\underline{\bm{\lambda}})$ for all $\underline{\bm{\lambda}}\in \mathcal{P}_{r,n}^{d}.$ For the content function on Std$(\underline{\bm{\lambda}})$ for all $\underline{\bm{\lambda}}\in \mathcal{P}_{r,n}^{d},$ we use the following elements:
$$C_{\mathfrak{t}}(k) =\begin{cases}\text{c}_{\mathfrak{t}}(k) & \hbox {if } k=1,2,\ldots,n; \\\zeta_{\text{p}_{\mathfrak{t}}^{r}(k-n)}& \hbox {if } k=n+1,n+2,\ldots, 2n.\end{cases}$$

By an argument which is similar to [ER, Lemma 19], we see that the elements $L_{n+k}$ $(1\leq k\leq n)$ satisfy the condition \eqref{defini7-2-identity} in Definition \ref{definition-7-2-jucys-murphy}. Thus, we only need to check the cases $k=1,2,\ldots, n.$

Suppose that $\underline{\bm{\lambda}}=(\bm{\lambda}^{(1)},\ldots,\bm{\lambda}^{(r)})=((\lambda_{1}^{(1)},\ldots,\lambda_{d}^{(1)}),\ldots,(\lambda_{1}^{(r)},\ldots,\lambda_{d}^{(r)}))\in \mathcal{P}_{r,n}^{d}$ and that $\mathfrak{t}^{\underline{\bm{\lambda}}}=(\mathfrak{t}^{\bm{\lambda}^{(1)}},\ldots,\mathfrak{t}^{\bm{\lambda}^{(r)}})
=((\mathfrak{t}^{\lambda_{1}^{(1)}},\ldots,\mathfrak{t}^{\lambda_{d}^{(1)}}),\ldots,
(\mathfrak{t}^{\lambda_{1}^{(r)}},\ldots,\mathfrak{t}^{\lambda_{d}^{(r)}})).$

Let us first consider the case when $\mathfrak{t}=((\mathfrak{t}_{1}^{(1)},\ldots,\mathfrak{t}_{d}^{(1)}),\ldots,(\mathfrak{t}_{1}^{(r)},\ldots,\mathfrak{t}_{d}^{(r)}))$ is a standard $(r,d)$-tableau of shape $\underline{\bm{\lambda}}$ and satisfies $\text{p}_{\mathfrak{t}}^{r}(i)=\text{p}_{\mathfrak{t}^{\underline{\bm{\lambda}}}}^{r}(i)$ for all $1\leq i\leq n.$ Suppose that $k$ is in the $(l, j)$-th component $\mathfrak{t}_{j}^{(l)}$ of $\mathfrak{t}.$ By \eqref{giXj}, we have
\begin{equation}\label{identify7-1}
m_{\mathfrak{t}^{\underline{\bm{\lambda}}}\mathfrak{t}}X_{k}=m_{\mathfrak{t}^{\bm{\lambda}^{(1)}}\mathfrak{t}^{(1)}}\cdots m_{\mathfrak{t}^{\bm{\lambda}^{(l)}}\mathfrak{t}^{(l)}}X_{k}m_{\mathfrak{t}^{\bm{\lambda}^{(l+1)}}\mathfrak{t}^{(l+1)}}\cdots m_{\mathfrak{t}^{\bm{\lambda}^{(r)}}\mathfrak{t}^{(r)}},
\end{equation}
where $\mathfrak{t}^{\bm{\lambda}^{(i)}}=(\mathfrak{t}^{\lambda_{1}^{(i)}},\ldots,\mathfrak{t}^{\lambda_{d}^{(i)}})$ and $\mathfrak{t}^{(i)}=(\mathfrak{t}_{1}^{(i)},\ldots,\mathfrak{t}_{d}^{(i)})$ for each $1\leq i\leq r.$

We can adapt the proof of [JM, Proposition 3.7] to our case, and we get that $m_{\mathfrak{t}^{\bm{\lambda}^{(l)}}\mathfrak{t}^{(l)}}X_{k}$ is equal to
\begin{equation}\label{identify7-2}
\text{c}_{\mathfrak{t}}(k)m_{\mathfrak{t}^{\bm{\lambda}^{(l)}}\mathfrak{t}^{(l)}}+\sum\limits_{\substack{\mathfrak{d}\in \mathrm{Std}(\bm{\lambda}^{(l)})\\\mathfrak{d}\rhd \mathfrak{t}^{(l)}}}a_{\mathfrak{d}}m_{\mathfrak{t}^{\bm{\lambda}^{(l)}}\mathfrak{d}}+\sum\limits_{\substack{\mathfrak{u},\mathfrak{v} \in \mathrm{Std}(\bm{\mu}^{(l)})\\\bm{\mu}^{(l)}\rhd \bm{\lambda}^{(l)} }}r_{\mathfrak{u}\mathfrak{v}}m_{\mathfrak{u}\mathfrak{v}},
\end{equation}
for some $a_{\mathfrak{d}}, r_{\mathfrak{u}\mathfrak{v}}\in \mathcal{R}.$

Set $\mathfrak{a} :=(\mathfrak{t}^{\bm{\lambda}^{(1)}},\ldots, \mathfrak{u},\ldots, \mathfrak{t}^{\bm{\lambda}^{(r)}}),$ $\mathfrak{b} :=(\mathfrak{t}^{(1)},\ldots, \mathfrak{v},\ldots, \mathfrak{t}^{(r)})$ and $\mathfrak{s} :=(\mathfrak{t}^{(1)},\ldots, \mathfrak{d},\ldots, \mathfrak{t}^{(r)}).$ Then $\mathfrak{s}\in \text{Std}(\underline{\bm{\lambda}}),$ $\mathfrak{a}, \mathfrak{b}\in \text{Std}(\underline{\bm{\mu}}),$ where $\underline{\bm{\mu}}=(\bm{\lambda}^{(1)},\ldots,\bm{\mu}^{(l)},\ldots,\bm{\lambda}^{(r)}).$ Moreover, by the definition of the dominance order, we have $\mathfrak{s}\rhd\mathfrak{t}$ and $\underline{\bm{\mu}}\rhd \underline{\bm{\lambda}}.$ Therefore, we have $m_{\mathfrak{a}\mathfrak{b}}\in Y_{r,n}^{d,\rhd \underline{\bm{\lambda}}},$ where $m_{\mathfrak{a}\mathfrak{b}}=m_{\mathfrak{t}^{\bm{\lambda}^{(1)}}\mathfrak{t}^{(1)}}\cdots m_{\mathfrak{u}\mathfrak{v}}\cdots m_{\mathfrak{t}^{\bm{\lambda}^{(r)}}\mathfrak{t}^{(r)}}$ and $Y_{r,n}^{d,\rhd \underline{\bm{\lambda}}}$ is the $\mathcal{R}$-submodule of $Y_{r,n}^{d}$ generated by the elements $\{m_{\mathfrak{g}\mathfrak{h}}\:|\:\mathfrak{g}, \mathfrak{h}\in \mathrm{Std}(\underline{\bm{\mu}})\text{ with }\underline{\bm{\mu}}\rhd \underline{\bm{\lambda}}\}.$

Multiplying the expression \eqref{identify7-2} on the left with $m_{\mathfrak{t}^{\lambda^{(1)}}\mathfrak{t}^{(1)}}\cdots m_{\mathfrak{t}^{\lambda^{(l-1)}}\mathfrak{t}^{(l-1)}}$ and on the right with $m_{\mathfrak{t}^{\lambda^{(l+1)}}\mathfrak{t}^{(l+1)}}\cdots m_{\mathfrak{t}^{\lambda^{(r)}}\mathfrak{t}^{(r)}}$, by \eqref{identify7-1}, we then get
\begin{equation}\label{identify7-3}
m_{\mathfrak{t}^{\underline{\bm{\lambda}}}\mathfrak{t}}X_{k}\equiv \text{c}_{\mathfrak{t}}(k)m_{\mathfrak{t}^{\underline{\bm{\lambda}}}\mathfrak{t}}+\sum\limits_{\substack{\mathfrak{s}\in \mathrm{Std}(\underline{\bm{\lambda}})\\\mathfrak{s}\rhd \mathfrak{t}}}a_{\mathfrak{s}}m_{\mathfrak{t}^{\underline{\bm{\lambda}}}\mathfrak{s}}
\quad\mathrm{mod}~Y_{r,n}^{d,\rhd \underline{\bm{\lambda}}},
\end{equation}
which shows that this proposition is true in this case.

For $\mathfrak{t}$ an arbitrary standard $(r,d)$-tableau of shape $\underline{\bm{\lambda}}$, we can write $\mathfrak{t}=\mathfrak{t}_{1}w,$ where $\mathfrak{t}_{1}$ is a standard $(r,d)$-tableau of shape $\underline{\bm{\lambda}}$ such that $\text{p}_{\mathfrak{t}_{1}}^{r}(i)=\text{p}_{\mathfrak{t}^{\underline{\bm{\lambda}}}}^{r}(i)$ for all $1\leq i\leq n$ and $w$ is a distinguished right coset representative of $\mathfrak{S}_{\beta}$ in $\mathfrak{S}_{n},$ where $\beta=(\beta_1,\ldots,\beta_r)$ with each $\beta_i=\sum_{l=1}^{d}|\lambda_{l}^{(i)}|$ for $1\leq i\leq r.$

Let $w=s_{i_1}s_{i_2}\cdots s_{i_t}$ be a reduced expression of $w$. Then we have $i_j$ and $i_j+1$ have different $r$-position in $\mathfrak{t}_{1}s_{i_1}\cdots s_{i_{j-1}}$ for all $j\geq 1$ and that $\mathfrak{t}_{1}s_{i_1}\cdots s_{i_{j-1}}s_{i_{j}}$ is obtained from $\mathfrak{t}_{1}s_{i_1}\cdots s_{i_{j-1}}$ by interchanging $i_j$ and $i_j+1.$ By Lemmas \ref{gxxg-xggx} and \ref{lemma6-16-lambd}, we get that
\begin{equation}\label{identify7-4}
m_{\mathfrak{t}^{\underline{\bm{\lambda}}}\mathfrak{t}}X_{k}=m_{\mathfrak{t}^{\underline{\bm{\lambda}}}\mathfrak{t}_1}g_{w}X_{k}
=m_{\mathfrak{t}^{\underline{\bm{\lambda}}}\mathfrak{t}_1}X_{(k)w^{-1}}g_w.
\end{equation}

By the choice of $\mathfrak{t}_{1},$ we can apply \eqref{identify7-3} and get that
\begin{align}\label{identify7-5}
m_{\mathfrak{t}^{\underline{\bm{\lambda}}}\mathfrak{t}}X_{k}&=m_{\mathfrak{t}^{\underline{\bm{\lambda}}}\mathfrak{t}_1}X_{(k)w^{-1}}g_w\notag\\
&\equiv\bigg(\text{c}_{\mathfrak{t}_1}((k)w^{-1})m_{\mathfrak{t}^{\underline{\bm{\lambda}}}\mathfrak{t}_1}+\sum\limits_{\substack{\mathfrak{s}_1\in \mathrm{Std}(\underline{\bm{\lambda}})\\\mathfrak{s}_1\rhd \mathfrak{t}_1}}a_{\mathfrak{s}_1}m_{\mathfrak{t}^{\underline{\bm{\lambda}}}\mathfrak{s}_1}\bigg)g_w
\quad\mathrm{mod}~Y_{r,n}^{d,\rhd \underline{\bm{\lambda}}}.
\end{align}
Since all the $\mathfrak{s}_1$ appearing in the summation of \eqref{identify7-5} satisfy $\text{p}_{\mathfrak{s}_{1}}^{r}(i)=\text{p}_{\mathfrak{t}^{\underline{\bm{\lambda}}}}^{r}(i)$ for all $1\leq i\leq n$, we have $m_{\mathfrak{t}^{\underline{\bm{\lambda}}}\mathfrak{s}_1}g_w=m_{\mathfrak{t}^{\underline{\bm{\lambda}}}\mathfrak{s}}$ for some $\mathfrak{s}\in \mathrm{Std}(\underline{\bm{\lambda}})$ with $\mathfrak{s}\rhd \mathfrak{t}$, and also $m_{\mathfrak{t}^{\underline{\bm{\lambda}}}\mathfrak{t}_1}g_{w}=m_{\mathfrak{t}^{\underline{\bm{\lambda}}}\mathfrak{t}}.$ Therefore, we have
\begin{align}\label{identify7-6}
m_{\mathfrak{t}^{\underline{\bm{\lambda}}}\mathfrak{t}}X_{k}
\equiv\text{c}_{\mathfrak{t}}(k)m_{\mathfrak{t}^{\underline{\bm{\lambda}}}\mathfrak{t}}+\sum\limits_{\substack{\mathfrak{s}\in \mathrm{Std}(\underline{\bm{\lambda}})\\\mathfrak{s}\rhd \mathfrak{t}}}a_{\mathfrak{s}}m_{\mathfrak{t}^{\underline{\bm{\lambda}}}\mathfrak{s}}
\quad\mathrm{mod}~Y_{r,n}^{d,\rhd \underline{\bm{\lambda}}},
\end{align}
where $a_{\mathfrak{s}}=a_{\mathfrak{s}_1}\in \mathcal{R}.$ We have proved this proposition.
\end{proof}

\subsection{Idempotents of $Y_{r, n}^{d}$}
Let $\mathbb{K}$ be an algebraically closed field of characteristic $p\geq 0$ such that $p$ does not divide $r$, which contains some invertible elements $v_{1},\ldots, v_{d}.$ In this subsection, we shall work with a specialised cyclotomic Yokonuma-Hecke algebra $Y_{r, n}^{d, \mathbb{K}}$ defined over $\mathbb{K},$ that is, $v_i\in \mathbb{K}^{*}$ for $1\leq i\leq d$ and $q\in \mathbb{K}^{*}.$

From Proposition \ref{proposition7-3-jm}, we can now apply the general theory developed in [Ma2] to obtain a sufficient condition for the semi-simplicity criterion of $Y_{r, n}^{d, \mathbb{K}}$. By [Ma2, Corollary 2.9], we get that $Y_{r, n}^{d, \mathbb{K}}$ is split semisimple if
\begin{equation}\label{semisimple}
\prod_{k=1}^{n}(1+q^{2}+\cdots+q^{2(k-2)}+q^{2(k-1)})\prod_{1\leq i< j\leq d}\prod_{-n< l< n}(q^{2l}v_i-v_j)\neq 0.
\end{equation}

\begin{remark}\label{remark-semisimple}
In fact, \eqref{semisimple} is also a necessary condition on the semisimplicity of $Y_{r, n}^{d, \mathbb{K}}.$
\end{remark}

From now on, we always assume that the condition \eqref{semisimple} holds. In particular, we can apply all the results in [Ma2, Section 3].

We shall follow the arguments of [Ma2, Section 3] to construct a ``seminormal" basis of $Y_{r, n}^{d, \mathbb{K}}$. For $1\leq k\leq n,$ we define the following two sets:
\[\mathcal{C}(k) :=\{\text{c}_{\mathfrak{t}}(k)\:|\:\mathfrak{t}\in \text{Std}(\underline{\bm{\lambda}})\text{ for some }\underline{\bm{\lambda}}\in \mathcal{P}_{r,n}^{d}\},\]
and
\[\overline{\mathcal{C}(k)} :=\{\zeta_{\text{p}_{\mathfrak{t}}^{r}(k)}\:|\:\mathfrak{t}\in \text{Std}(\underline{\bm{\lambda}})\text{ for some }\underline{\bm{\lambda}}\in \mathcal{P}_{r,n}^{d}\}.\]

\begin{definition}\label{definition-idempotent}
Suppose that $\underline{\bm{\lambda}}\in \mathcal{P}_{r,n}^{d}$ and that $\mathfrak{s}, \mathfrak{t}\in \text{Std}(\underline{\bm{\lambda}}).$

(i) Let
\begin{equation}\label{idempotent-ET}
E_{\mathfrak{t}}=\prod_{k=1}^{n}\bigg(\prod_{\substack{c\in \mathcal{C}(k)\\c\neq \text{c}_{\mathfrak{t}}(k)}}\frac{X_{k}-c}{\text{c}_{\mathfrak{t}}(k)-c}\cdot \prod_{\substack{\bar{c}\in \overline{\mathcal{C}(k)}\\\bar{c}\neq \zeta_{\text{p}_{\mathfrak{t}}^{r}(k)}}}\frac{t_{k}-\bar{c}}{\zeta_{\text{p}_{\mathfrak{t}}^{r}(k)}-\bar{c}}
\bigg).
\end{equation}

(ii) Let $e_{\mathfrak{s}\mathfrak{t}}^{\underline{\bm{\lambda}}}=E_{\mathfrak{s}}\hspace{0.3mm}m_{\mathfrak{s}\mathfrak{t}}\hspace{0.3mm}E_{\mathfrak{t}}.$
\end{definition}

By Proposition \ref{proposition7-3-jm}, we can now apply the general theory developed in [Ma2, Section 3] to get the following results.

\begin{proposition}\label{idempotent-property}
$(\emph{i})$ The set $\{e_{\mathfrak{s}\mathfrak{t}}^{\underline{\bm{\lambda}}}\:|\:\mathfrak{s}, \mathfrak{t}\in \emph{Std}(\underline{\bm{\lambda}})\text{ for some }\underline{\bm{\lambda}}\in \mathcal{P}_{r,n}^{d}\}$ is a $\mathbb{K}$-basis of $Y_{r, n}^{d, \mathbb{K}}.$

$(\emph{ii})$ For $\underline{\bm{\lambda}}, \underline{\bm{\mu}}\in \mathcal{P}_{r,n}^{d}$ and $\mathfrak{s}, \mathfrak{t}\in \emph{Std}(\underline{\bm{\lambda}}),$ $\mathfrak{u}, \mathfrak{v}\in \emph{Std}(\underline{\bm{\mu}}),$ we have
\begin{equation}\label{result-1}
e_{\mathfrak{s}\mathfrak{t}}^{\underline{\bm{\lambda}}}\hspace{0.3mm}X_k=\emph{c}_{\mathfrak{t}}(k)e_{\mathfrak{s}\mathfrak{t}}^{\underline{\bm{\lambda}}},\hspace{1cm}
e_{\mathfrak{s}\mathfrak{t}}^{\underline{\bm{\lambda}}}\hspace{0.3mm}t_k=\zeta_{\emph{p}_{\mathfrak{t}}^{r}(k)}e_{\mathfrak{s}\mathfrak{t}}^{\underline{\bm{\lambda}}},\hspace{1cm} e_{\mathfrak{s}\mathfrak{t}}^{\underline{\bm{\lambda}}}\hspace{0.3mm}E_{\mathfrak{u}}=\delta_{\mathfrak{t}, \mathfrak{u}}e_{\mathfrak{s}\mathfrak{u}}^{\underline{\bm{\lambda}}},
\end{equation}
and moreover, there exists a scalar $0\neq \gamma_{\mathfrak{t}}\in \mathbb{K}$ such that
\begin{align}\label{result-2}
e_{\mathfrak{s}\mathfrak{t}}^{\underline{\bm{\lambda}}}\hspace{0.3mm}e_{\mathfrak{u}\mathfrak{v}}^{\underline{\bm{\mu}}}=
\begin{cases}
\gamma_{\mathfrak{t}}\hspace{0.3mm}e_{\mathfrak{s}\mathfrak{v}}^{\underline{\bm{\lambda}}} & \text{if } \underline{\bm{\lambda}}=\underline{\bm{\mu}}\text{ and }\mathfrak{t}=\mathfrak{u};
\\
\hspace{0.25cm}0 & \text{otherwise.}
\end{cases}
\end{align}
In particular, $\gamma_{\mathfrak{t}}$ depends only on $\mathfrak{t}$ and the set $\{e_{\mathfrak{s}\mathfrak{t}}^{\underline{\bm{\lambda}}}\:|\:\mathfrak{s}, \mathfrak{t}\in \emph{Std}(\underline{\bm{\lambda}})\text{ and }\underline{\bm{\lambda}}\in \mathcal{P}_{r,n}^{d}\}$ is a cellular basis of $Y_{r, n}^{d, \mathbb{K}}.$

$(\emph{iii})$ For $\underline{\bm{\lambda}}\in \mathcal{P}_{r,n}^{d}$ and $\mathfrak{t}\in \emph{Std}(\underline{\bm{\lambda}}),$ we have $E_{\mathfrak{t}}=\frac{1}{\gamma_{\mathfrak{t}}}\hspace{0.3mm}e_{\mathfrak{t}\mathfrak{t}}^{\underline{\bm{\lambda}}}.$ Moreover, these elements $\{E_{\mathfrak{t}}\:|\:\mathfrak{t}\in \emph{Std}(\underline{\bm{\lambda}})\text{ for some }\underline{\bm{\lambda}}\in \mathcal{P}_{r,n}^{d}\}\}$ give a complete set of pairwise orthogonal primitive idempotents for $Y_{r, n}^{d, \mathbb{K}}.$

$(\emph{iv})$ For $\underline{\bm{\lambda}}\in \mathcal{P}_{r,n}^{d}$ and $\mathfrak{t}\in \emph{Std}(\underline{\bm{\lambda}}),$ we have
\begin{equation}\label{result-3}
E_{\mathfrak{t}}\hspace{0.3mm}X_k=X_k\hspace{0.3mm}E_{\mathfrak{t}}=\emph{c}_{\mathfrak{t}}(k)E_{\mathfrak{t}},\hspace{1cm}
E_{\mathfrak{t}}\hspace{0.3mm}t_k=t_k\hspace{0.3mm}E_{\mathfrak{t}}=\zeta_{\emph{p}_{\mathfrak{t}}^{r}(k)}E_{\mathfrak{t}}.
\end{equation}

$(\emph{v})$ The Jucys-Murphy elements $X_1,\ldots,X_n,$ $t_1,\ldots,t_n$ generate a maximal commutative subalgebra of $Y_{r, n}^{d, \mathbb{K}}.$
\end{proposition}

\section{Appendix. Fusion procedure for cyclotomic Yokonuma-Hecke algebras}

In this appendix, inspired by [PA], we prove tha the primitive idempotents of cyclotomic Yokonuma-Hecke algebras can be constructed by consecutive evaluations of a certain rational function.

Jucys [Juc] has claimed that the primitive idempotents of symmetric groups $\mathfrak{S}_n$ indexed by standard Young tableaux can be obtained by taking a certain limiting process on a rational function, which is now commonly referred to as the fusion procedure. The procedure has been further developed in the situation of Hecke algebras [Ch], see also [Na1-3]. Molev [Mo] has proposed an alternative approach of the fusion procedure for the symmetric group, which relies on the existence of a maximal commutative subalgebra generated by the Jucys-Murphy elements. Here the idempotents are obtained by consecutive evaluations of a certain rational function. The simplified version of the fusion procedure has been generalized to the Hecke algebras of type $A$ [IMO], to the Brauer algebras [IM, IMOg1], to the Birman-Murakami-Wenzl algebras [IMOg2], to the complex reflection groups of type $G(d,1,n)$ [OgPA1], to the Ariki-Koike algebras [OgPA2], to the wreath products of finite groups by the symmetric group [PA], to the degenerate cyclotomic Hecke algebras [ZL] and to the Yokonuma-Hecke algebras [C].

In this section, we continue to consider a split semisimple cyclotomic Yokonuma-Hecke algebra $Y_{r, n}^{d, \mathbb{K}}$ over $\mathbb{K};$ that is, we always assume that $v_i\in \mathbb{K}^{*}$ ($1\leq i\leq d$) and $q\in \mathbb{K}^{*}$ satisfy the condition \eqref{semisimple}.

\subsection{Inductive formulae of $E_{\mathfrak{t}}$} We first introduce two rational functions, and then present the inductive formulae of the primitive idempotents $E_{\mathfrak{t}}$ defined in \eqref{idempotent-ET}.

For $\lambda, \mu$ two partitions and a node $\theta=(x,y)\in [\lambda],$ we define the hook length $h_{\lambda}(\theta)$ and the generalized hook length $h_{\lambda}^{\mu}(\theta)$ of $\theta$ with respect to $(\lambda, \mu)$, respectively:
\begin{equation}\label{hook-length1}
h_{\lambda}(\theta) :=\lambda_{x}+\lambda'_{y}-x-y+1,
\end{equation}
and
\begin{equation}\label{hook-length2}
h_{\lambda}^{\mu}(\theta) :=\lambda_{x}+\mu'_{y}-x-y+1.
\end{equation}

Let $\underline{\bm{\lambda}}=((\lambda_{1}^{(1)},\ldots,\lambda_{d}^{(1)}),\ldots,(\lambda_{1}^{(r)},\ldots,\lambda_{d}^{(r)}))$ be an $(r,d)$-partition and $\bm{\theta}=(\theta, k,l)=((x,y),k,l)$ an $(r,d)$-node of $[\underline{\bm{\lambda}}].$ We define the hook length $h_{\underline{\bm{\lambda}}}(\bm{\theta})$ of $\bm{\theta}$ in $\underline{\bm{\lambda}}$ to be the hook length of the node $\theta$ in the partition of $\underline{\bm{\lambda}}$ with position $(k,l)$, that is,
\begin{equation}\label{hook-length3}
h_{\underline{\bm{\lambda}}}(\bm{\theta}) :=h_{\lambda^{(k)}_{l}}(\theta)=\lambda^{(k)}_{l,x}+\lambda^{(k)'}_{l,y}-x-y+1.
\end{equation}
Let $\mu$ be another partition. We define the generalized hook length $h_{\underline{\bm{\lambda}}}^{\mu}(\bm{\theta})$ of $\bm{\theta}$ with respect to $(\underline{\bm{\lambda}}, \mu)$ to be the generalized hook length of $\theta$ with respect to $(\lambda^{(k)}_{l}, \mu),$ that is,
\begin{equation}\label{hook-length4}
h_{\underline{\bm{\lambda}}}^{\mu}(\bm{\theta}) :=h_{\lambda^{(k)}_{l}}^{\mu}(\theta)=\lambda^{(k)}_{l,x}+\mu'_{y}-x-y+1.
\end{equation}

Let $S=\{\zeta_{1},\zeta_{2},\ldots,\zeta_{r}\}$ be the set of all $r$-th roots of unity. For an $(r,d)$-partition $\underline{\bm{\lambda}}=((\lambda_{1}^{(1)},\ldots,\lambda_{d}^{(1)}),\ldots,(\lambda_{1}^{(r)},\ldots,\lambda_{d}^{(r)}))$, we define
\begin{equation}\label{f-lambda-t}
\text{F}_{\underline{\bm{\lambda}}}^{T} :=\prod_{\bm{\theta}\in \underline{\bm{\lambda}}} \Big(\prod_{\substack{\xi\in S\\\xi\neq \zeta_{\text{p}^{(r)}(\bm{\theta})}}}(\zeta_{\text{p}^{(r)}(\bm{\theta})}-\xi)\Big),
\end{equation}
and
\begin{equation}\label{f-lambda}
\text{F}_{\underline{\bm{\lambda}}} :=\prod_{\bm{\theta}\in \underline{\bm{\lambda}}} \bigg(\frac{[h_{\underline{\bm{\lambda}}}(\bm{\theta})]_{q}}{q^{\text{cc}(\bm{\theta})}}
\prod_{\substack{1\leq k\leq d\\k\neq \text{p}^{(d)}(\bm{\theta})}}\frac{v_{\text{p}^{(d)}(\bm{\theta})}
q^{h_{\underline{\bm{\lambda}}}^{\lambda^{(\text{p}^{(r)}(\bm{\theta}))}_{k}}(\bm{\theta})}-v_{k}
q^{-h_{\underline{\bm{\lambda}}}^{\lambda^{(\text{p}^{(r)}(\bm{\theta}))}_{k}}(\bm{\theta})}}{q^{-\text{cc}(\bm{\theta})}}\bigg),
\end{equation}
where $[a]_{q}=q^{a-1}+q^{a-3}+\cdots+q^{-a+1}$ for $a\in \mathbb{Z}_{\geq 0}.$

The following lemma can be easily proved.
\begin{lemma}\label{equality-Flambda}
When $r=1$ and $d=m,$ $\emph{F}_{\underline{\bm{\lambda}}}$ is equal to $\emph{F}_{\bm{\lambda}}^{-1}$ defined in $[\emph{OgPA}2, \emph{Sect.} 2.2(12)].$
\end{lemma}
\begin{remark}\label{remark1}
When $r=1$ and $d=l,$ the cyclotomic Yokonuma-Hecke algebra $\mathrm{Y}_{r,n}^{d}$ is exactly the cyclotomic Hecke algebra $H_{n}^{l},$ whose Schur element has been explicitly calculated in [ChJa, Theorem 3.2(2)] (see also [GIM] and [Ma]). By comparing \eqref{f-lambda} with the formula in [ChJa, Theorem 3.2(2)], it is easy to see that $\text{F}_{\underline{\bm{\lambda}}}$ in \eqref{f-lambda} is proportional to the Schur element described in [ChJa, Theorem 3.2(2)].
\end{remark}

Let $\underline{\bm{\lambda}}=((\lambda_{1}^{(1)},\ldots,\lambda_{d}^{(1)}),\ldots,(\lambda_{1}^{(r)},\ldots,\lambda_{d}^{(r)}))\in \mathcal{P}_{r,n}^{d}.$ For an $(r,d)$-node $\bm{\theta}=(\theta, k, l)\in [\underline{\bm{\lambda}}],$ we denote by $[\underline{\bm{\lambda}}\setminus\bm{\theta}]$ the set of $(r,d)$-nodes after removing $\theta$ from $[\lambda_{l}^{(k)}]$ and $[\underline{\bm{\lambda}}\cup\bm{\theta}]$ the set of $(r,d)$-nodes after adding $\theta$ to $[\lambda_{l}^{(k)}].$ We then call $\bm{\theta}$ removable from $\underline{\bm{\lambda}}$ if $\underline{\bm{\lambda}}\setminus\bm{\theta}$ is an $(r,d)$-partition and addable to $\underline{\bm{\lambda}}$ if $\underline{\bm{\lambda}}\cup\bm{\theta}$ is an $(r,d)$-partition, respectively.

For each $\underline{\bm{\lambda}}\in \mathcal{P}_{r,n}^{d}$ and $\mathfrak{t}\in \text{Std}(\underline{\bm{\lambda}}),$ we denote by $\text{c}(\mathfrak{t}|i)$ and $\text{p}^{(r)}(\mathfrak{t}|i)$ the quantum content and the $r$-position of the $(r,d)$-node with the integer $i$ in it, respectively. For simplicity, we set
\begin{equation}\label{ci-pi-123}
\text{c}_i :=\text{c}(\mathfrak{t}|i) \quad\text{ and } \quad\text{p}_i :=\text{p}^{(r)}(\mathfrak{t}|i) \quad\text{ for }i=1,\ldots,n.
\end{equation}
We then define
\begin{equation}\label{f-TT}
\text{F}_{\mathfrak{t}}^{T}(\underline{v}) :=\prod_{\substack{\xi\in S\\\xi\neq \zeta_{\text{p}_{n}}}}\frac{1}{\underline{v}-\xi},
\end{equation}
and
\begin{equation}\label{f-T}
\text{F}_{\mathfrak{t}}(u) :=\frac{u-\text{c}_n}{(u-v_1)\cdots (u-v_d)}\prod_{i=1}^{n-1}\frac{(u-\text{c}_i)^{2}}{(u-\text{c}_i)^{2}-(q-q^{-1})^{2}u\text{c}_i\delta_{\text{p}_i, \text{p}_n}},
\end{equation}
where $\delta_{\text{p}_i, \text{p}_n}$ is the Kronecker delta.

we denote by $\bm{\theta}$ the $(r,d)$-node of $\mathfrak{t}$ containing the number $n.$ Since $\mathfrak{t}$ is standard, $\bm{\theta}$ is removable from $\underline{\bm{\lambda}}.$ Let $\mathfrak{u}$ be the standard $(r,d)$-tableau obtained from $\mathfrak{t}$ by removing $\bm{\theta}$ and let $\underline{\bm{\mu}}$ be the shape of $\mathfrak{u}$.

Recall that for any fixed $r$-th root of unity $\xi$, we have
\begin{equation}
\label{young-module-5-19}\prod_{\xi\neq \beta\in S}(\xi-\beta)=r\xi^{-1}.
\end{equation} Then $\text{F}_{\mathfrak{t}}^{T}(\underline{v})$ is non-singular at $\underline{v}=\zeta_{\text{p}_{n}},$ and moreover, from \eqref{f-lambda-t} we have
\begin{equation}\label{f-T-relation}
\text{F}_{\mathfrak{t}}^{T}(\underline{v})\Big|_{\underline{v}=\zeta_{\text{p}_{n}}}=\frac{\zeta_{\text{p}_{n}}}{r}=
(\text{F}_{\underline{\bm{\lambda}}}^{T})^{-1}\text{F}_{\underline{\bm{\mu}}}^{T}.
\end{equation}

The following proposition can be proved in exactly the same way as in [OgPA1, Proposition 3.4 and OgPA2, Proposition 3].
\begin{proposition}
The rational function $\emph{F}_{\mathfrak{t}}(u)$ is non-singular at $u=\emph{c}_n$, and moreover, we have
\begin{equation}\label{f-T-u}
\emph{F}_{\mathfrak{t}}(u)\Big|_{u=\emph{c}_n}=\emph{F}_{\underline{\bm{\lambda}}}^{-1}\emph{F}_{\underline{\bm{\mu}}}.
\end{equation}
\end{proposition}

Denote by $\mathcal{E}_{+}(\underline{\bm{\mu}})$ the set of $(r,d)$-nodes addable to $\underline{\bm{\mu}}.$ By \eqref{idempotent-ET}, We can rewrite $E_{\mathfrak{t}}$ inductively as follows:
\begin{equation}\label{inductive-formula}
E_{\mathfrak{t}}=E_{\mathfrak{u}}\prod_{\substack{\bm{\theta'}\in \mathcal{E}_{+}(\underline{\bm{\mu}})\\\text{c}(\bm{\theta'})\neq\text{c}(\bm{\theta})}}
\frac{X_{n}-\text{c}(\bm{\theta'})}{\text{c}(\bm{\theta})-\text{c}(\bm{\theta'})}\prod_{\substack{\bm{\theta'}\in \mathcal{E}_{+}(\underline{\bm{\mu}})\\\text{p}^{(r)}(\bm{\theta'})\neq\text{p}^{(r)}(\bm{\theta})}}
\frac{t_{n}-\zeta_{\text{p}^{(r)}(\bm{\theta'})}}{\zeta_{\text{p}^{(r)}(\bm{\theta})}-\zeta_{\text{p}^{(r)}(\bm{\theta'})}}
\end{equation}
with $E_{\mathfrak{t}_{0}}=1$ for the unique standard $(r,d)$-tableau $\mathfrak{t}_{0}$ of size $0.$

Assume that $\{\mathfrak{t}_1,\ldots,\mathfrak{t}_k\}$ is the set of pairwise different standard $(r,d)$-tableaux obtained from $\mathfrak{u}$ by adding an $(r,d)$-node containing the number $n.$ Notice that $\mathfrak{t}\in \{\mathfrak{t}_1,\ldots,\mathfrak{t}_k\}.$ Moreover, by branching properties, we have
\begin{equation}\label{sum-formula}
E_{\mathfrak{u}}=\sum_{i=1}^{k}E_{\mathfrak{t}_{i}}.
\end{equation}
We consider the following rational function in $u$ and $\underline{v}$:
\begin{equation}\label{rational-function}
\frac{u-\text{c}_n}{u-X_{n}}\frac{\underline{v}-\zeta_{\text{p}_n}}{\underline{v}-t_{n}}E_{\mathfrak{u}}.
\end{equation}
The formulae \eqref{result-3} imply that \eqref{rational-function} is non-singular at $u=\text{c}_n$ and $\underline{v}=\zeta_{\text{p}_n}$. Moreover, if we replace $E_{\mathfrak{u}}$ with the right-hand side of \eqref{sum-formula}, we can easily get
\begin{equation}\label{sum-function}
\frac{u-\text{c}_n}{u-X_{n}}\frac{\underline{v}-\zeta_{\text{p}_n}}{\underline{v}-t_{n}}E_{\mathfrak{u}}\Big|_{\underline{v}=\zeta_{\text{p}_{n}}}\Big|_{u=\text{c}_n}=E_{\mathfrak{t}}.
\end{equation}

\subsection{Fusion formulae for $E_{\mathfrak{t}}$}
We first define a rational function in variables $a,b$ with values in $\mathrm{Y}_{r, n}^{d}$ as follows:
\begin{equation}\label{Baxter-element}
g_{i}(a,b) :=g_{i}+(q-q^{-1})\frac{be_{i}}{a-b}\quad\mbox{for}~i=1,\ldots,n-1.
\end{equation}
The following lemma is proved in [C, Lemma 2.1].
\begin{lemma}\label{Bax-elements}
The rational functions $g_{i}(a,b)$ satisfy the following relations$:$
\begin{equation}\label{Baxter-element1}
g_{i}(a,b)g_{i+1}(a,c)g_{i}(b,c)=g_{i+1}(b,c)g_{i}(a,c)g_{i+1}(a,b)\quad\mbox{for}~i=1,\ldots,n-1,
\end{equation}
\begin{equation}\label{Baxter-element2}
\hspace*{20pt}g_{i}(a,b)g_{i}(b,a)=1-(q-q^{-1})^{2}\frac{abe_{i}}{(a-b)^{2}}\qquad\mbox{for}~i=1,\ldots,n-1.
\end{equation}
\end{lemma}

Recall that $S=\{\zeta_1,\ldots,\zeta_r\}$ is the set of all $r$-th roots of unity. We set
\begin{equation}\label{gamma-function}
\Gamma(\underline{v_1},\ldots,\underline{v_n}) :=\prod_{i=1}^{n}\Big(\frac{\Pi_{\xi\in S}(\underline{v_i}-\xi)}{\underline{v_i}-t_i}\Big).
\end{equation}

Let $\phi_{1}(u) :=\frac{(u-v_1)\cdots (u-v_d)}{u-X_1}.$ For $k=2,\ldots,n$, we set
\begin{align}\label{phi-function}
\phi_k(u_1,\ldots,u_{k-1},u)& :=g_{k-1}(u,u_{k-1})\phi_{k-1}(u_1,\ldots,u_{k-2},u)g_{k-1}^{-1}\notag\\
&=g_{k-1}(u,u_{k-1})g_{k-2}(u,u_{k-2})\cdots g_{1}(u,u_{1})\phi_{1}(u)\cdot g_{1}^{-1}\cdots g_{k-1}^{-1}.
\end{align}

We now define the following element:
\begin{equation}\label{e-upn}
E_{\mathfrak{u}, \text{p}_{n}} :=\frac{\underline{v}-\zeta_{\text{p}_n}}{\underline{v}-t_{n}}E_{\mathfrak{u}}\Big|_{\underline{v}=\zeta_{\text{p}_{n}}}.
\end{equation}
By definition, $E_{\mathfrak{u}, \text{p}_{n}}$ is an idempotent which is equal to the sum of the idempotents $E_{\mathfrak{s}}$, where $\mathfrak{s}$ runs through the set of standard $(r,d)$-tableaux obtained from $\mathfrak{u}$ by adding an $(r,d)$-node $\bm{\theta}$ containing the integer $n$ and satisfying $\text{p}^{(r)}(\bm{\theta})=\text{p}_{n}.$

\begin{lemma}
Assume that $n\geq 1.$ We have
\begin{align}\label{F-PhiEu}
\emph{F}_{\mathfrak{t}}(u)\phi_{n}(\emph{c}_1,\ldots,\emph{c}_{n-1},u)E_{\mathfrak{u}, \emph{p}_{n}}=\frac{u-\emph{c}_n}{u-X_{n}}E_{\mathfrak{u}, \emph{p}_{n}}.
\end{align}
\end{lemma}
\begin{proof}
We shall prove the lemma by induction on $n.$

When $n=1,$ we can write the left-hand side of \eqref{F-PhiEu} as follows: $$\frac{u-\text{c}_1}{(u-v_1)\cdots (u-v_d)}\phi_{1}(u)E_{\mathfrak{u}, \text{p}_{n}}=\frac{u-\text{c}_1}{u-X_1}E_{\mathfrak{u}, \text{p}_{n}}.$$

Assume that $n> 1$. Note that the left-hand side of \eqref{F-PhiEu} is equal to
\[\text{F}_{\mathfrak{t}}(u)\Big(g_{n-1}+(q-q^{-1})\frac{\text{c}_{n-1}e_{n-1}}{u-{\text{c}_{n-1}}}\Big)\cdots \Big(g_{1}+(q-q^{-1})\frac{\text{c}_{1}e_{1}}{u-{\text{c}_{1}}}\Big)\phi_{1}(u)\cdot g_{1}^{-1}\cdots g_{n-1}^{-1}E_{\mathfrak{u}, \text{p}_{n}}.\]
For $k=1,\ldots,n-1$, by \eqref{relations}, we have
\begin{align*}
e_{k}&\Big(g_{k-1}+(q-q^{-1})\frac{\text{c}_{k-1}e_{k-1}}{u-{\text{c}_{k-1}}}\Big)\cdots
\Big(g_{1}+(q-q^{-1})\frac{\text{c}_{1}e_{1}}{u-{\text{c}_{1}}}\Big)\\
&=\Big(g_{k-1}+(q-q^{-1})\frac{\text{c}_{k-1}e_{k-1}}{u-{\text{c}_{k-1}}}\Big)\cdots
\Big(g_{1}+(q-q^{-1})\frac{\text{c}_{1}e_{1}}{u-{\text{c}_{1}}}\Big)\cdot e_{1,k+1},
\end{align*}
and $e_{1,k+1}\cdot g_{1}^{-1}\cdots g_{n-1}^{-1}=g_{1}^{-1}\cdots g_{n-1}^{-1}\cdot e_{k,n}.$

Note that $e_{k,n} E_{\mathfrak{u}, \text{p}_{n}}=0$ if $\text{p}_{k}\neq\text{p}_{n}$. Thus, we can write the left-hand side of \eqref{F-PhiEu} as follows:
\begin{align}\label{Fg-PhiEu}
\text{F}_{\mathfrak{t}}(u)&\Big(g_{n-1}+(q-q^{-1})\frac{\delta_{\text{p}_{n-1}, \text{p}_{n}}\text{c}_{n-1}e_{n-1}}{u-{\text{c}_{n-1}}}\Big)\cdots\notag\\ &\times\Big(g_{1}+(q-q^{-1})\frac{\delta_{\text{p}_{1}, \text{p}_{n}}\text{c}_{1}e_{1}}{u-{\text{c}_{1}}}\Big)\phi_{1}(u)\cdot g_{1}^{-1}\cdots g_{n-1}^{-1}E_{\mathfrak{u}, \text{p}_{n}}.
\end{align}

Suppose first that $\text{p}_i\neq \text{p}_n$ for $i=1,\ldots,n-1.$ In this situation, due to \eqref{Fg-PhiEu}, we have the left-hand side of \eqref{F-PhiEu} is equal to
\begin{align}\label{Fg-PhiEuphi}
\text{F}_{\mathfrak{t}}(u)g_{n-1}\cdots &g_{1}\phi_{1}(u)\cdot g_{1}^{-1}\cdots g_{n-1}^{-1}E_{\mathfrak{u}, \text{p}_{n}}\notag\\
&=(u-\text{c}_{n})g_{n-1}\cdots g_{1}(u-X_1)^{-1}g_{1}^{-1}\cdots g_{n-1}^{-1}E_{\mathfrak{u}, \text{p}_{n}}.
\end{align}
Note that
\begin{align*}
&g_{1}(u-X_1)g_{1}^{-1}(u-X_2)^{-1}g_{2}^{-1}\cdots g_{n-1}^{-1}E_{\mathfrak{u}, \text{p}_{n}}\\
=&\big(u-g_1X_1(g_1-(q-q^{-1})e_1)\big)(u-X_2)^{-1}g_{2}^{-1}\cdots g_{n-1}^{-1}E_{\mathfrak{u}, \text{p}_{n}}\\
=&(u-X_2)(u-X_2)^{-1}g_{2}^{-1}\cdots g_{n-1}^{-1}E_{\mathfrak{u}, \text{p}_{n}}+(q-q^{-1})g_1X_1(u-X_2)^{-1}e_1g_{2}^{-1}\cdots g_{n-1}^{-1}E_{\mathfrak{u}, \text{p}_{n}}\\
=&g_{2}^{-1}\cdots g_{n-1}^{-1}E_{\mathfrak{u}, \text{p}_{n}}+(q-q^{-1})g_1X_1(u-X_2)^{-1}g_{2}^{-1}\cdots g_{n-1}^{-1}e_{1,n}E_{\mathfrak{u}, \text{p}_{n}}\\
=&g_{2}^{-1}\cdots g_{n-1}^{-1}E_{\mathfrak{u}, \text{p}_{n}}.
\end{align*}
Therefore, we have $$g_{1}(u-X_1)^{-1}g_{1}^{-1}\cdots g_{n-1}^{-1}E_{\mathfrak{u}, \text{p}_{n}}=(u-X_2)^{-1}g_{2}^{-1}\cdots g_{n-1}^{-1}E_{\mathfrak{u}, \text{p}_{n}}.$$
By repeating the process above on the right-hand side of \eqref{Fg-PhiEuphi}, we finally get that the left-hand side of \eqref{F-PhiEu} is equal to
$$(u-\text{c}_{n})(u-X_n)^{-1}E_{\mathfrak{u}, \text{p}_{n}},$$
which is exactly the right-hand side of \eqref{F-PhiEu}.

Next suppsoe that there exists some $l\in \{1,\ldots,n-1\}$ such that $\text{p}_{l}=\text{p}_{n}.$ We shall fix the unique $l$ such that $\text{p}_{l}=\text{p}_{n}$ and $\text{p}_{i}\neq \text{p}_{n}$ for $i=l+1,\ldots,n-1.$

Assume that $\mathfrak{v}$ is the standard $(r,d)$-tableau obtained from $\mathfrak{u}$ by removing the $(r,d)$-nodes containing the integers $l+1,\ldots,n-1$ and $\mathfrak{w}$ is the standard $(r,d)$-tableau obtained from $\mathfrak{v}$ by removing the $(r,d)$-node containing the integer $l.$ We then define
\begin{equation*}
E_{\mathfrak{w}, \text{p}_{l}} :=\frac{\underline{v}-\zeta_{\text{p}_l}}{\underline{v}-t_{l}}E_{\mathfrak{w}}\Big|_{\underline{v}=\zeta_{\text{p}_{l}}}.
\end{equation*}

Since $E_{\mathfrak{w}}$ can be expressed in terms of $X_1,\ldots,X_{l-1}$ and $t_1,\ldots,t_{l-1}$, we see that $E_{\mathfrak{w}}$ commutes with $g_{l}^{-1}g_{l+1}^{-1}\cdots g_{n-1}^{-1}$. Note that $E_{\mathfrak{w}}E_{\mathfrak{u}}=E_{\mathfrak{u}}=E_{\mathfrak{u}}^{2},$ $E_{\mathfrak{u}, \text{p}_{n}}^{2}=E_{\mathfrak{u}, \text{p}_{n}}$, $\text{p}_{l}=\text{p}_{n}$ and $t_lg_{l}^{-1}\cdots g_{n-1}^{-1}$$=g_{l}^{-1}\cdots g_{n-1}^{-1}t_n.$ Thus, we get
\begin{align}\label{EW-PhiEu}
E_{\mathfrak{w}, \text{p}_{l}}g_{l}^{-1}g_{l+1}^{-1}\cdots g_{n-1}^{-1}E_{\mathfrak{u}, \text{p}_{n}}&=g_{l}^{-1}\cdots g_{n-1}^{-1}\frac{\underline{v}-\zeta_{\text{p}_l}}{\underline{v}-t_{n}}E_{\mathfrak{w}}E_{\mathfrak{u}, \text{p}_{n}}\Big|_{\underline{v}=\zeta_{\text{p}_{l}}}\notag\\
&=g_{l}^{-1}\cdots g_{n-1}^{-1}E_{\mathfrak{u}, \text{p}_{n}}.
\end{align}
By \eqref{EW-PhiEu}, we can rewrite \eqref{Fg-PhiEu} as follows:
\begin{align*}
\text{F}_{\mathfrak{t}}(u)g_{n-1}\cdots g_{l+1}\Big(g_{l}+(q-q^{-1})\frac{\text{c}_{l}e_{l}}{u-{\text{c}_{l}}}\Big)\phi_{l}(\emph{c}_1,\ldots,\emph{c}_{l-1},u)E_{\mathfrak{w}, \text{p}_{l}}g_{l}^{-1}\cdots g_{n-1}^{-1}E_{\mathfrak{u}, \emph{p}_{n}}.
\end{align*}
By the induction hypothesis, we have
\begin{align*}
\phi_{l}(\text{c}_1,\ldots,\text{c}_{l-1},u)E_{\mathfrak{w}, \text{p}_{l}}=\text{F}_{\mathfrak{v}}(u)^{-1}\frac{u-\text{c}_l}{u-X_{l}}E_{\mathfrak{w}. \text{p}_{l}},
\end{align*}
Now we can use \eqref{EW-PhiEu} again to get that the left-hand side of \eqref{F-PhiEu} can be written as
\begin{align}\label{FF-PhiEu}
\text{F}_{\mathfrak{t}}(u)\text{F}_{\mathfrak{v}}(u)^{-1}g_{n-1}\cdots g_{l+1}\Big(g_{l}+(q-q^{-1})\frac{\text{c}_{l}e_{l}}{u-{\text{c}_{l}}}\Big)\frac{u-\text{c}_l}{u-X_{l}}g_{l}^{-1}\cdots g_{n-1}^{-1}E_{\mathfrak{u}, \emph{p}_{n}}.
\end{align}

Since $X_{n}$ commutes with $E_{\mathfrak{u}, \emph{p}_{n}},$ we can move $(u-X_n)^{-1}$ from the right-hand side of \eqref{F-PhiEu} to the left-hand side. By \eqref{inverse} and the fact that $e_kg_{k+1}\cdots g_{n-1}=g_{k+1}\cdots g_{n-1}e_{k,n}$ and $e_{k,n} E_{\mathfrak{u}, \text{p}_{n}}=0$ for $k=l+1,\ldots,n-1,$ it is easy to see that we can move $g_{n-1}\cdots g_{l+1}$ to the right-hand side. By \eqref{Baxter-element2}, $g_{l}(u, \text{c}_l)$ is invertible. Finally we get that \eqref{F-PhiEu} is equivalent to the following equality:
\begin{align}\label{Equi-equality}
\text{F}_{\mathfrak{t}}(u)&\text{F}_{\mathfrak{v}}(u)^{-1}(u-\text{c}_l)g_{l}^{-1}\cdots g_{n-1}^{-1}(u-X_n)E_{\mathfrak{u}, \text{p}_{n}}
=(u-\text{c}_n)(u-X_l)\notag\\
&\times\Big(g_{l}+(q-q^{-1})\frac{ue_{l}}{\text{c}_{l}-u}\Big)
\Big(1-(q-q^{-1})^{2}\frac{u\text{c}_le_{l}}{(u-\text{c}_{l})^{2}}\Big)^{-1}g_{l+1}\cdots g_{n-1}E_{\mathfrak{u}, \text{p}_{n}}.
\end{align}

Since $\text{p}_{l}=\text{p}_{n}$ and $\text{p}_{i}\neq \text{p}_{n}$ for $i=l+1,\ldots,n-1,$ we have, by the definition \eqref{f-T}, that
\begin{align*}
\text{F}_{\mathfrak{t}}(u)\text{F}_{\mathfrak{v}}(u)^{-1}=\frac{u-\text{c}_n}{u-\text{c}_l}\frac{(u-\text{c}_l)^{2}}{(u-\text{c}_l)^{2}-(q-q^{-1})^{2}u\text{c}_l}.
\end{align*}
Notice that $e_lg_{l+1}\cdots g_{n-1}=g_{l+1}\cdots g_{n-1}e_{l,n}.$ Since $\text{p}_{l}=\text{p}_{n},$ we have $e_{l,n}E_{\mathfrak{u}, \text{p}_{n}}=E_{\mathfrak{u}, \text{p}_{n}}.$ Therefore, to verify \eqref{Equi-equality}, it suffices to prove that
\begin{align}\label{Equi-equality1}
g_{l}^{-1}\cdots g_{n-1}^{-1}(u-X_n)E_{\mathfrak{u}, \text{p}_{n}}
=(u-X_l)\Big(g_{l}+(q-q^{-1})\frac{ue_{l}}{\text{c}_{l}-u}\Big)g_{l+1}\cdots g_{n-1}E_{\mathfrak{u}, \text{p}_{n}}.
\end{align}

By \eqref{JM-elements}, we get that $g_{l}^{-1}g_{l+1}^{-1}\cdots g_{n-1}^{-1}X_n=X_lg_{l}g_{l+1}\cdots g_{n-1}$, and we can write the left-hand side of \eqref{Equi-equality1} as
\begin{align}\label{Equi-equality2}
ug_{l}^{-1}\cdots g_{n-1}^{-1}E_{\mathfrak{u}, \text{p}_{n}}-X_lg_{l}\cdots g_{n-1}E_{\mathfrak{u}, \text{p}_{n}}.
\end{align}
By the fact that $e_{l,n} E_{\mathfrak{u}, \text{p}_{n}}=E_{\mathfrak{u}, \text{p}_{n}}$ and $e_{k,n} E_{\mathfrak{u}, \text{p}_{n}}=0$ for $k=l+1,\ldots,n-1,$ we can write \eqref{Equi-equality2} as follows:
\begin{align}\label{Equi-equality3}
&ug_{l}^{-1}g_{l+1}\cdots g_{n-1}E_{\mathfrak{u}, \text{p}_{n}}-X_lg_{l}\cdots g_{n-1}E_{\mathfrak{u}, \text{p}_{n}}\notag\\
=&u\big(g_l-(q-q^{-1})e_l\big)g_{l+1}\cdots g_{n-1}E_{\mathfrak{u}, \text{p}_{n}}-X_lg_{l}\cdots g_{n-1}E_{\mathfrak{u}, \text{p}_{n}}\notag\\
=&(u-X_l)g_{l}\cdots g_{n-1}E_{\mathfrak{u}, \text{p}_{n}}-(q-q^{-1})ug_{l+1}\cdots g_{n-1}E_{\mathfrak{u}, \text{p}_{n}}.
\end{align}

By definition, we have $X_lE_{\mathfrak{u}, \text{p}_{n}}=\text{c}_lE_{\mathfrak{u}, \text{p}_{n}}.$ Moreover, $X_l$ commutes with $g_{l+1}\cdots g_{n-1}$ by \eqref{giXj}. Therefore, we can write the right-hand side of \eqref{Equi-equality1} as
\begin{align}\label{Equi-equality4}
&(u-X_l)g_{l}\cdots g_{n-1}E_{\mathfrak{u}, \text{p}_{n}}+(q-q^{-1})(u-X_l)\frac{ue_{l}}{\text{c}_{l}-u}g_{l+1}\cdots g_{n-1}E_{\mathfrak{u}, \text{p}_{n}}\notag\\
=&(u-X_l)g_{l}\cdots g_{n-1}E_{\mathfrak{u}, \text{p}_{n}}+(q-q^{-1})(u-X_l)\frac{u}{\text{c}_{l}-u}g_{l+1}\cdots g_{n-1}E_{\mathfrak{u}, \text{p}_{n}}\notag\\
=&(u-X_l)g_{l}\cdots g_{n-1}E_{\mathfrak{u}, \text{p}_{n}}+(q-q^{-1})\cdot\frac{u}{\text{c}_{l}-u}g_{l+1}\cdots g_{n-1}(u-X_l)E_{\mathfrak{u}, \text{p}_{n}}\notag\\
=&(u-X_l)g_{l}\cdots g_{n-1}E_{\mathfrak{u}, \text{p}_{n}}+(q-q^{-1})\cdot\frac{u}{\text{c}_{l}-u}g_{l+1}\cdots g_{n-1}(u-\text{c}_l)E_{\mathfrak{u}, \text{p}_{n}}\notag\\
=&(u-X_l)g_{l}\cdots g_{n-1}E_{\mathfrak{u}, \text{p}_{n}}-(q-q^{-1})ug_{l+1}\cdots g_{n-1}E_{\mathfrak{u}, \text{p}_{n}}.
\end{align}

Comparing \eqref{Equi-equality4} with \eqref{Equi-equality3}, we get that \eqref{Equi-equality1} holds.
\end{proof}

Recall the function $\phi_k(u_1,\ldots,u_{k-1},u)$ defined in \eqref{phi-function}. For each $k=1,\ldots,n,$ we define
\begin{equation}\label{tilde-phi}
\widetilde{\phi}_{k}(u_1,\ldots,u_{k-1},u,\underline{v}) :=\phi_k(u_1,\ldots,u_{k-1},u)\cdot\Big(\frac{\Pi_{\xi\in S}(\underline{v}-\xi)}{\underline{v}-t_k}\Big),
\end{equation}
and the following rational function:
\begin{align}\label{Phi-function}
\Phi(u_1,\ldots,u_n,\underline{v_1},\ldots,\underline{v_n}) :=&\phi_n(u_1,\ldots,u_{n})\phi_{n-1}(u_1,\ldots,u_{n-1})\notag\\
&\cdots \phi_1(u_1)\Gamma(\underline{v_1},\ldots,\underline{v_n}).
\end{align}

\begin{lemma}
Assume that $n\geq 1.$ We have
\begin{align}\label{FF-Phi}
\emph{F}_{\mathfrak{t}}^{T}(\underline{v})\emph{F}_{\mathfrak{t}}(u)\widetilde{\phi}_{n}(\emph{c}_1,\ldots,\emph{c}_{n-1},u,\underline{v})
E_{\mathfrak{u}}\Big|_{\underline{v}=\zeta_{\emph{p}_{n}}}=\frac{u-\emph{c}_n}{u-X_{n}}
\frac{\underline{v}-\zeta_{\emph{p}_n}}{\underline{v}-t_{n}}E_{\mathfrak{u}}\Big|_{\underline{v}=\zeta_{\emph{p}_{n}}}.
\end{align}
\end{lemma}
\begin{proof}
By \eqref{f-TT}, we have
\begin{align}\label{Ftt-equality}
\emph{F}_{\mathfrak{t}}^{T}(\underline{v})\cdot\Big(\frac{\Pi_{\xi\in S}(\underline{v}-\xi)}{\underline{v}-t_n}\Big)=\frac{\underline{v}-\zeta_{\text{p}_n}}{\underline{v}-t_{n}}.
\end{align}
By \eqref{e-upn}, \eqref{tilde-phi} and \eqref{Ftt-equality}, we see that \eqref{FF-Phi} is a direct consequence of \eqref{F-PhiEu}.
\end{proof}

Now we can state the main result of this paper.

\begin{theorem}\label{main-theorem}
The idempotent $E_{\mathfrak{t}}$ of $\mathrm{Y}_{r,n}^{d}$ corresponding to the standard $(r,d)$-tableau $\mathfrak{t}$ can be derived by the following consecutive evaluations$:$
\begin{equation}\label{idempotents}
E_{\mathfrak{t}}=\frac{1}{\emph{F}_{\underline{\bm{\lambda}}}^{T}\emph{F}_{\underline{\bm{\lambda}}}}
\Phi(u_1,\ldots,u_{n},\underline{v_1},\ldots,\underline{v_n})
\Big|_{\underline{v_1}=\zeta_{\emph{p}_{1}}}\cdots\Big|_{\underline{v_n}=\zeta_{\emph{p}_{n}}}\Big|_{u_{1}=\emph{c}_1}\cdots\Big|_{u_{n}=\emph{c}_{n}}.
\end{equation}
\end{theorem}
\begin{proof}
Since $g_i$ commutes with $t_k$ if $i< k-1$, we can rewrite $\Phi(u_1,\ldots,u_n,v_1,$$\ldots,v_n)$ as follows:
\begin{align}\label{Phitilde}
\Phi(&u_1,\ldots,u_n,\underline{v_1},\ldots,\underline{v_n})\notag\\
&=\widetilde{\phi}_{n}(u_1,\ldots,u_{n},\underline{v_n})\widetilde{\phi}_{n-1}(u_1,\ldots,u_{n-1},\underline{v_{n-1}})\cdots \widetilde{\phi}_{1}(u_1,\underline{v_1}).
\end{align}

We shall prove this theorem by induction on $n.$ For $n=0,$ the situation is trivial.

For $n> 0,$ by \eqref{Phitilde} and the induction hypothesis we can rewrite the right-hand side of \eqref{idempotents} as follows:
\begin{align}\label{Phitilde-4}
(\text{F}_{\underline{\bm{\lambda}}}^{T}\text{F}_{\underline{\bm{\lambda}}})^{-1}\text{F}_{\underline{\bm{\mu}}}^{T}\text{F}_{\underline{\bm{\mu}}}
\widetilde{\phi}_{n}(\text{c}_1,\ldots,\text{c}_{n-1},u_n,\underline{v_n})E_{\mathfrak{u}}\Big|_{\underline{v_n}=\zeta_{\text{p}_{n}}}\Big|_{u_n=\text{c}_{n}}.
\end{align}
By \eqref{FF-Phi} we can rewrite the expression \eqref{Phitilde-4} as follows:
\begin{align}\label{Phitilde-10}
(\text{F}_{\underline{\bm{\lambda}}}^{T}\text{F}_{\underline{\bm{\lambda}}})^{-1}\text{F}_{\underline{\bm{\mu}}}^{T}\text{F}_{\underline{\bm{\mu}}}
(\text{F}_{\mathfrak{t}}^{T}(\underline{v_n})\text{F}_{\mathfrak{t}}(u_n))^{-1}
\frac{u_n-\text{c}_n}{u_n-X_{n}}\frac{\underline{v_n}-\zeta_{\text{p}_n}}{\underline{v_n}-t_{n}}
E_{\mathfrak{u}}\Big|_{\underline{v_n}=\zeta_{\text{p}_{n}}}\Big|_{u_n=\text{c}_{n}}.
\end{align}

By \eqref{f-T-relation} and \eqref{f-T-u}, together with \eqref{sum-function} and \eqref{Phitilde-10}, we see that the right-hand side of \eqref{idempotents} is equal to $E_{\mathfrak{t}}.$
\end{proof}

Finally, let us consider an example.

\begin{example}
Assume that $r=d=2, n=4$ and $\underline{\bm{\lambda}}=(((2),(0)),((1),(1)))$ is a $(2,2)$-partition of $4$. We shall consider the following standard $(2,2)$-tableau of shape $\underline{\bm{\lambda}}$:
\[\mathfrak{t}=\left(\left(\hspace{-0.1cm}\begin{array}{l}\fbox{1}\fbox{3}\\[-0.05em] \end{array}
\, ,\, \varnothing\right)
\, ,\,\left(\hspace{-0.1cm}\begin{array}{l}\fbox{2}\\[-0.05em] \end{array}\,,\,\begin{array}{l}\fbox{4}\\[-0.05em] \end{array}\hspace{-0.1cm}\right)\right).\]

Theorem \ref{main-theorem} implies that the idempotent $E_{\mathfrak{t}}$ can be expressed as
\begin{align*}
E_{\mathfrak{t}}=&\frac{\zeta_{1}^{2}\zeta_{2}^{2}}{16(q+q^{-1})(v_1-v_2)(v_2q-v_1q^{-1})(v_1q-v_2q^{-1})^{2}}\\
&\times g_{3}(v_2,v_1q^{2})g_{2}(v_2,v_1)g_{1}(v_2,v_1)\phi_{1}(v_2)g_{1}^{-1}g_{2}^{-1}g_{3}^{-1}\\
&\times g_{2}(v_1q^{2}, v_1)g_{1}(v_1q^{2}, v_1)\phi_{1}(v_{1}q^{2})g_{1}^{-1}g_{2}^{-1}\times g_{1}(v_1,v_1)\phi_{1}(v_1)g_{1}^{-1}\\
&\times\phi_{1}(v_1)(\zeta_1+t_1)(\zeta_2+t_2)(\zeta_1+t_3)(\zeta_2+t_4).
\end{align*}
\end{example}

\vskip3mm Acknowledgements. The author is grateful to Professor G. Lusztig for pointing out the reference [Lu] to him.

%/////////////////////////////////////////////////////////////////////////////////////////////////////////////////////////////////////////////////////////

%by (\ref{stn})
%\begin{align*}
%  t_{n+1}&\geq\frac{1}{t_n+(1-t_n)\frac{1}{c}}\\
%  1-t_{n+1}&\leq \frac{(1-t_n)(\frac{1}{c}-1)}{t_n+(1-t_n)\frac{1}{c}}=\frac{(1-t_n)(\frac{1}{c}-1)}{(1-t_n)(\frac{1}{c}-1)+1}\\
%  \frac{1}{1-t_{n+1}}&\geq \frac{1}{(1-t_n)(\frac{1}{c}-1)+1}
%\end{align*}
%Since $c\geq \frac{1}{2}$, $\frac{1}{c}-1\geq 1$, we have
%$$\frac{1}{1-t_{n+1}}\geq \frac{1}{1-t_n}+1\geq\cdots\geq\frac{1}{1-t_1}+n\geq n+1$$
%i.e. $1-t_{n+1}\leq \frac{1}{n+1}$
%$$0\leq v_{2n+1}-x^*\leq v_{2n+1}-v_{2n}\leq v_{2n+1}-t_nv_{2n+1}\geq \frac{1}{n}v_{2n+1}$$
%So $\|v_{2n+1}-x^*\|\leq \frac{N}{n}\|\bar{v}\|$, where $N$ is the normal constant of $P$.
%/////////////////////////////////////////////////////////////////////////////////////////////////////////////////////////////////

\end{document}